\documentclass[12pt]{elsarticle}
\usepackage{graphicx} 
\usepackage{stmaryrd}
\usepackage{bbold}
\usepackage{subcaption}
\usepackage{geometry}
\usepackage{xcolor} 
\usepackage[normalem]{ulem} 
\usepackage{enumerate}
\usepackage{float}
\usepackage{mathrsfs}
\usepackage{booktabs} 
\usepackage{siunitx}  
\usepackage{amsthm,amsmath,amsfonts,amssymb}

\makeatletter
\def\ps@pprintTitle{%
 \let\@oddhead\@empty
 \let\@evenhead\@empty
 \def\@oddfoot{\hfill\@date}%
 \let\@evenfoot\@empty
}
\makeatother

\definecolor{darkorange}{rgb}{0.9, 0.43, 0.1}

\newtheorem{theorem}{Theorem}[section]
\newtheorem{lemma}{Lemma}[section]
\newtheorem{remark}{Remark}[section]

\numberwithin{equation}{section}

\newtheoremstyle{noparens}%
  {}{}%
  {\itshape}{}%
  {\bfseries}{.}%
  { }%
  {\thmname{#1}\thmnumber{ #2}\mdseries\thmnote{ #3}}

\theoremstyle{noparens}

\begin{document}

\begin{frontmatter}

\title{Maximum Bound Principle and Bound Preserving ETD schemes for a Phase-Field Model of Tumor Growth with Extracellular Matrix Degradation}

\author[BJUT]{Qiumei Huang}
\ead{qmhuang@bjut.edu.cn}

\author[hkpu]{Zhonghua Qiao}
\ead{zqiao@polyu.edu.hk}

\author[UMass]{Cheng Wang}
\ead{cwang1@umassd.edu}

\author[BJUT]{Huiting Yang \corref{cor1}}
\ead{yanghuiting2021@emails.bjut.edu.cn}

\address[BJUT]{College of Mathematics, Faculty of Science,
Beijing University of Technology,
Beijing 100124, China}
\address[hkpu]{Department of Applied Mathematics, The Hong Kong Polytechnic University,
Hung Hom, Kowloon, Hong Kong}
\address[UMass]{Department of Mathematics, University of Massachusetts Dartmouth, North Dartmouth, MA 02747}

\cortext[cor1]{Corresponding author}

\begin{abstract}
  In cancer research, the role of the extracellular matrix (ECM) and its associated matrix-degrading enzyme (MDE) has been a significant area of focus. This study presents a numerical algorithm designed to simulate a previously established tumor model that incorporates various biological factors, including tumor cells, viable cells, necrotic cells, and the dynamics of MDE and ECM. The model consists of a system that includes a phase field equation, two reaction-diffusion equations, and two ordinary differential equations. We employ the fast exponential time differencing Runge-Kutta (ETDRK) method with stabilizing terms to solve this system, resulting in a decoupled, explicit, linear numerical algorithm. The objective of this algorithm is to preserve the physical properties of the model variables, including the maximum bound principle (MBP) for nutrient concentration and MDE volume fraction, as well as bound preserving for ECM density and tumor volume fraction. We perform simulations of 2D and 3D tumor models {and discuss how different biological components impact growth dynamics. These simulations may help predict tumor evolution trends, offer insights for related biological and medical research,} potentially reduce the number and cost of experiments, and improve research efficiency.
\end{abstract}


\date{\today}

\end{frontmatter}

\section{Introduction}

In 2020, there were an estimated 19.3 million new cancer cases and nearly 10 million cancer deaths worldwide. By 2040, the global cancer burden is expected to reach 28.4 million cases, an increase $47\%$ from 2020 \cite{sung2021global}. These statistics underscore the importance of cancer research.

The extracellular matrix (ECM) is a complex network of large molecules (such as collagen, enzymes, and glycoproteins) secreted extracellularly, serving primarily as a structural scaffold and biochemical support for cells and tissues. However, the ECM in solid tumors exhibits significant differences from that in normal organs. The ECM often becomes disordered in cases of carcinogenesis. This abnormal ECM can influence cancer progression by directly promoting cell transformation and metastasis \cite{lu2012extracellular}. Over the past three to four decades, there has been an exponential increase in the study and recognition of the importance of the matrix in cancer \cite{cox2021matrix}. Numerous investigations have demonstrated that the ECM associated with tumors actively contributes to tumor cell growth, invasion, metastasis, and angiogenesis. Beyond influencing tumor cell behavior, the ECM also indirectly affects drug delivery mechanisms \cite{henke2020extracellular}. Cancer cells can secrete matrix-degrading enzymes (MDEs) that break down the ECM, allowing cancer cells to breach tissue barriers and spread locally, a critical process in cancer cell invasion. Consequently, matrix degradation is recognized as a fundamental mechanism in tumor development \cite{madsen2015source}. In recent years, many scientists have devoted considerable efforts to incorporating the ECM into mathematical models of tumors. A reaction-diffusion partial differential equation system was used in \cite{chaplain2011mathematical} to simulate the interactions among cancer cells, the ECM, and MDE. A system of five coupled partial differential equations was established by \cite{nguyen2018mathematical} to describe the dynamics and interactions of cancer cells, collagen fibers, and the enzymes MMP and lysyl oxidase (LOX), with the objective of better understanding cancer cell migration through ECM remodeling. A multiscale hybrid system composed of partial differential equations and stochastic differential equations was constructed by \cite{sfakianakis2020hybrid}, describing the evolution of epithelial cells (ECs) and mesenchymal-like cells (MCs), to simulate the combined invasion of the ECM by two types of cancer cells (the ECs and the MCs).
 Tumor growth towards lower extracellular matrix conductivity regions under Darcy’s law and steady morphology was discussed in \cite{zheng2022tumor} as well.

Phase-field equations have become increasingly popular in mathematical modeling of tumors, offering advantages such as the ability to describe complex interface dynamics and handle multiscale problems. Wu et al. \cite{wu2014stabilized} considered a diffuse-interface tumor growth system consisting of a reactive Cahn-Hilliard equation and a reaction-diffusion equation. They employed finite element spatial discretization and a modification of the Crank-Nicolson method as time discretization.  Xu et al. \cite{xu2017full} considered a model composed of two coupled compartments: cellular growth model using the phase-field approach, and the chemical transport model using standard diffusion-reaction equations. The isogeometric analysis (IGA) is used as the spatial discretization, while a generalized-$\alpha$ method is employed as the temporal discretization. The phase field Cahn-Hilliard equation and tumor growth models were simulated in \cite{mohammadi2019simulation}, using a second-order semi-implicit backward differential formula with a stabilized term in the time discretization, and the element-free Galerkin method for spatial discretization. In this research area, significant advancements have been made by J. Tinsley Oden and collaborators. Their work involves phase-field modeling, considering various factors influencing tumor growth \cite{fritz2021analysis,fritz2019local,fritz2019unsteady,hawkins2012numerical,lima2014hybrid}. These studies have employed finite element methods for spatial discretization and semi-implicit algorithms for time discretization.

In recent years, the numerical design has increasingly focused on preserving the physical properties of the original models. The preservation of the  MBP and bound preserving has attracted significant attention from researchers, making it a hot research topic in the field of computational mathematics.  Ju et al. \cite{ju2021maximum} developed time integration schemes for semilinear parabolic equations that preserve the high-order  MBP, utilizing the integrating factor Runge-Kutta (IFRK) method. Feng et al. \cite{feng2021maximum} presented a modified Leap-Frog finite difference scheme for the Allen-Cahn equation, which is linear second-order, maximum-principle preserving, and unconditionally energy-stable. Furthermore, for Allen-Cahn equations, Zhang et al. \cite{zhang2023third,zhang2023temporal} introduced explicit, large time-stepping algorithms that maintain fixed-points and maximum principle for any time step size, and a high-order maximum-principle-preserving framework. In particular, Du et al. \cite{du2019maximum,du2021maximum} introduced a framework that allows for the analysis of the MBP in the context of semilinear parabolic equations and proposed schemes for exponential time differencing that unconditionally preserve the MBP. Given the importance of physical structure-preserving properties in the numerical simulations, this study also aims to construct schemes that preserve the MBP and bound preserving specifically for tumor growth models. In our previous research \cite{huang2024maximum}, we successfully developed a second-order exponential time-differencing Runge-Kutta method (ETDRK2) that ensures the non-negativity of the tumor sequence parameter and two concentration variables. The constructed algorithm effectively simulates the growth of prostate cancer (PCa) and its response to drug treatment.

The primary focus of this study encompasses several aspects. We propose efficient, explicit, linear, and decoupled ETD schemes for the tumor growth model, which are easy to solve and enable an effective use of fast algorithms for 2D and 3D tumor models.
Additionally, we prove the numerical scheme also adheres to the MBP and boundedness as much as possible. {Furthermore, we also simulated tumor evolution under different initial ECM conditions and analyzed how various biological components affect the growth dynamics of the system through different parameter settings. These simulation results contribute to a deeper understanding of the tumor growth process and offer insights that may assist in predicting tumor growth.
}

The structure of this paper is organized as follows. In Section 2, the tumor growth model is reviewed. In Section 3, we introduce fully discrete schemes and prove that these schemes satisfy the discrete MBP and ensure discrete bound preservation. In Section 4, the numerical simulations of 2D and 3D tumor growth models are presented. Finally, the main achievements of this study is summarized, and potential future research directions are discussed in Section 5.


\section{Model review}

This section provides an overview of the mathematical model of tumor growth under investigation, providing a phase-field approach to tumor growth that incorporates the degradation of the ECM. This model, a simplified version of the original model in Fritz et al. \cite{fritz2019local}, employs a fourth-order phase-field equation to capture tumor dynamics. It also includes two reaction-diffusion equations for critical nutrients and MDE. The behavior of necrotic tumor cells and the ECM are represented through two separate ordinary differential equations.

\subsection{Total tumor cells dynamics}

In the area of tumor modeling, tumors are conventionally segregated into three distinct regions \cite{fritz2023tumor}, as illustrated in Fig.~\ref{fig:tumor}(\subref{tumor3phase}).  These regions encompass the proliferative region, the hypoxic region, and the necrotic region, represented by the volume fractions $\phi_P$, $\phi_H$, and $\phi_N$, respectively. Therefore, the total volume fraction of tumor cells, $\phi_T$, can be expressed as the sum of these three components, that is, $\phi_T=\phi_P+\phi_H+\phi_N$. As illustrated in Fig.~\ref{fig:tumor}(\subref{tumor3cell}), the outer layer, known as the proliferative region, is characterized by cells that have a high propensity to mitosis due to sufficient nutrients, thus fueling tumor growth. The middle layer, or the hypoxic region, is populated by cells that are in a state of hypoxia due to nutrient insufficiency, rendering them incapable of proliferation. The innermost layer, called the necrotic region, is where cells have succumbed to nutrient deficiency, primarily due to the difficulty in nutrient penetration into the tumor core. Considering that both the proliferative and hypoxic regions consist of viable cells, we are able to express the volume fraction of these viable tumor cells, $\phi_V$, as the sum of $\phi_P$ and $\phi_H$, that is, $\phi_V=\phi_T-\phi_N$. The aforementioned modeling method for tumors is highly popular in the numerical simulations, providing a clearer understanding of the structure and composition of the tumor.
\begin{figure}[t!]
  \centering
  \begin{subfigure}{0.36\textwidth}
      \includegraphics[width=\textwidth]{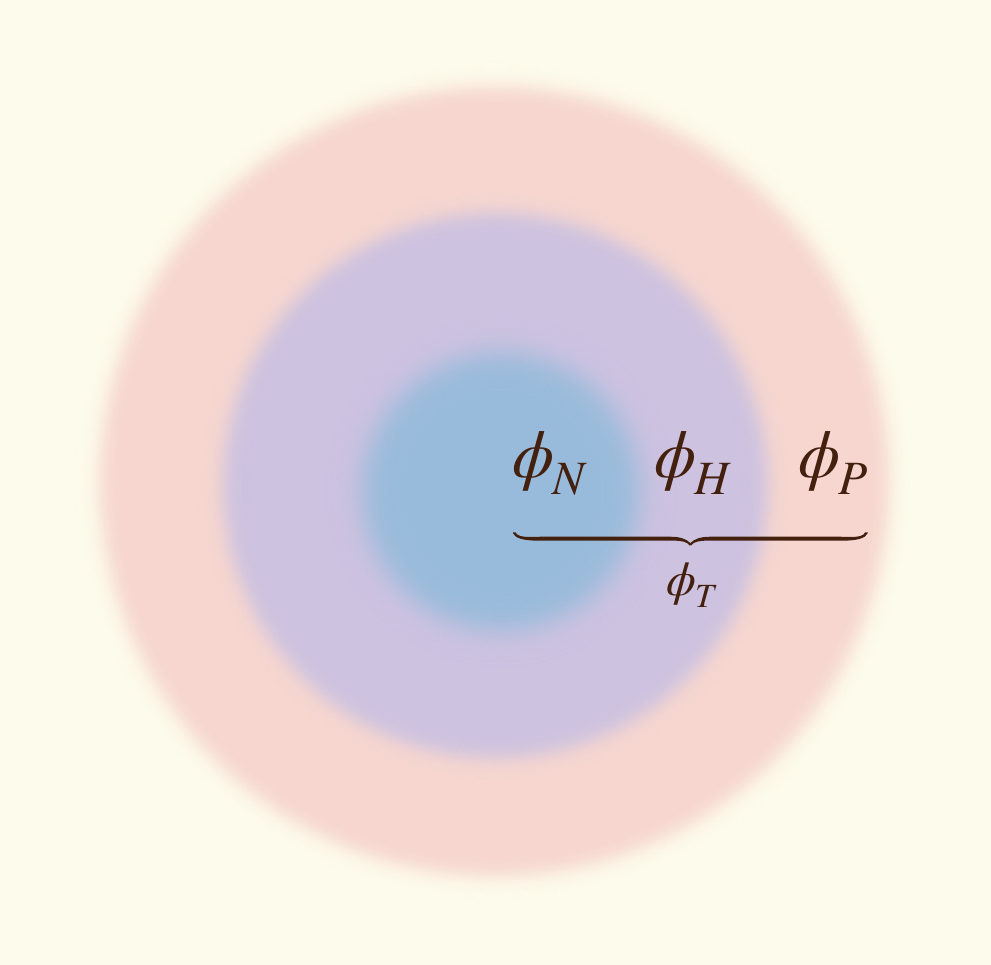}
      \caption{}
      \label{tumor3phase}
  \end{subfigure}
  \begin{subfigure}{0.54\textwidth}
      \includegraphics[width=\textwidth]{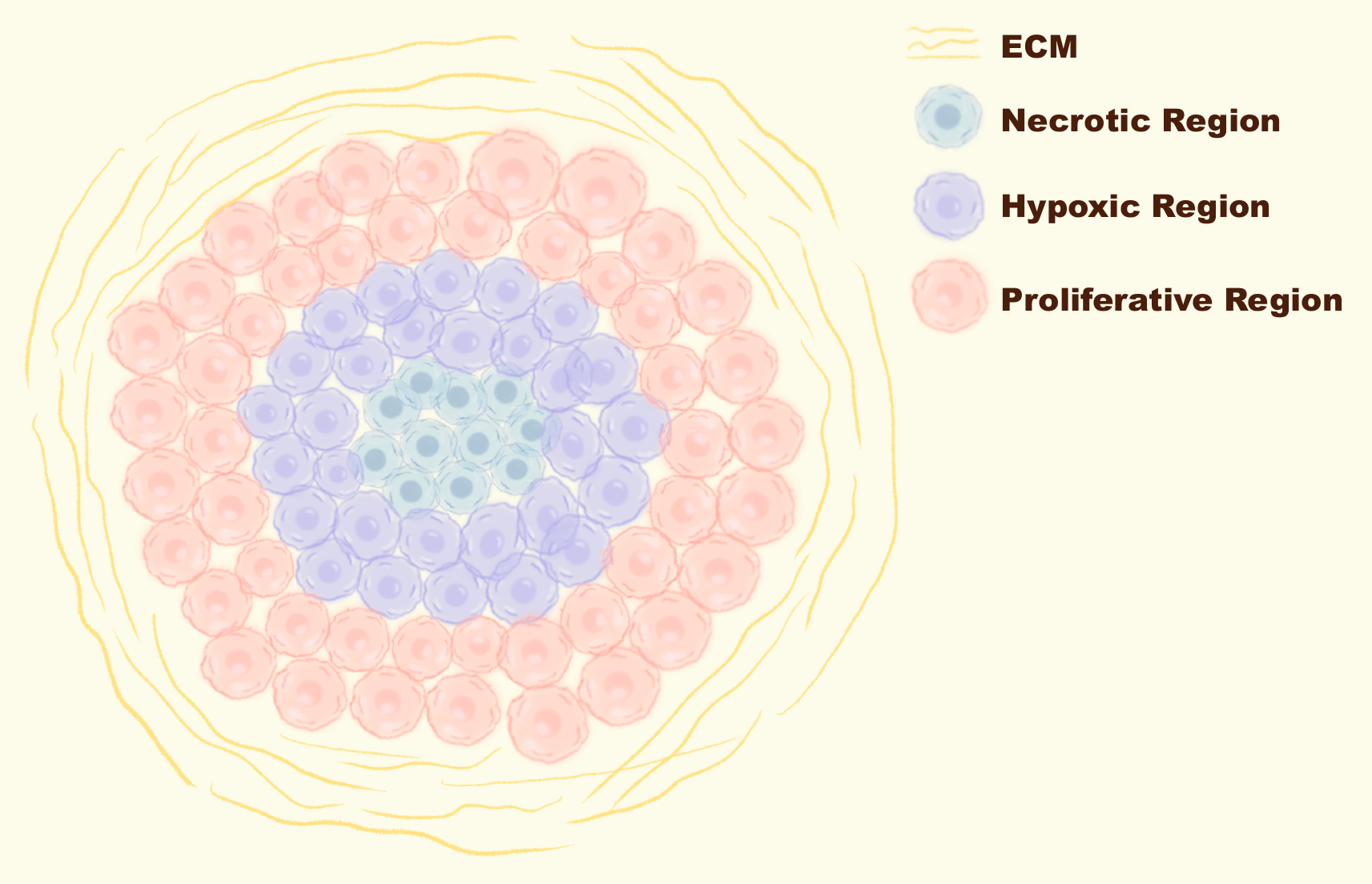}
      \caption{}
      \label{tumor3cell}
  \end{subfigure}
  \caption{Organization of tumor cells and phase of tumor dynamics}
  \label{fig:tumor}
\end{figure}

The dynamics of $\phi_T$ are dictated by several factors. The growth of the tumor adheres to a logistic model, with a proliferation rate $\lambda_T^\mathrm{pro}$. The volume fraction of total tumor cells can experience a decrease due to the natural apoptosis of viable cells, occurring at a rate denoted by $\lambda_T^\mathrm{apo}$, which is a non-negative constant. Consequently, the evolution equation for $\phi_T$ is expressed as the following phase field equation:
\begin{equation*}
\partial_t \phi_T=\operatorname{div} J-\operatorname{div} J_\alpha+\lambda_T^{\mathrm{pro}} \phi_\sigma \phi_V\left(1-\phi_T\right)-\lambda_T^{\mathrm{apo}} \phi_V.
\end{equation*}

In this equation, $J$ is the mass flux and $J_\alpha$ is the adhesion flux, representing the influx of tumor mass due to the haptotaxis effect. The mass flux $J$ turns out to be
$$
  J=m_T\left(\phi_V\right) \nabla \mu,
$$
where $m_T\left(\phi_V\right)$ is the cell mobility matrix defined as $m_T\left(\phi_V\right)=M_T \phi_V^2\left(1-\phi_V\right)^2$ with $M_T>0$. The chemical potential $\mu$ is defined as
\begin{equation*}
\mu=\Psi^{\prime}\left(\phi_T\right)-\varepsilon_T^2 \Delta \phi_T,
\end{equation*}
in which $\Psi$ is a double-well potential, given by $\Psi\left(\phi_T\right)=\bar{E} \phi_T^2\left(1-\phi_T\right)^2$ with $\bar{E}>0$, and $\varepsilon_T$ is a positive parameter associated with surface energy.

Finally, the adhesion flux, denoted as $J_\alpha$, represents a local gradient-based haptotaxis effect, and is defined as
\begin{equation*}
J_\alpha\left(\phi_V, \theta\right)=\chi_H \phi_V  \nabla \theta,
\end{equation*}
where $\chi_H$ represents the haptotaxis parameter and $\theta$ is the density of ECM (discussed in Sec.~\ref{secECM}).

\subsection{Necrotic tumor cell dynamics}

{The volume fraction of necrotic tumor cells, $\phi_N$, is assumed to be non-diffusive, meaning that it evolves over time but does not spread through spatial diffusion.} Its dynamics are described by the following equation:
\begin{equation}\label{phiN}
  \partial_t \phi_N=\lambda_{V N} \mathcal{H}\left(\sigma_{V N}-\phi_\sigma\right) \phi_V.
  \end{equation}
  {Here, $\phi_\sigma$ represents the nutrient concentration, which varies dynamically in both space and time, due to the diffusion process and consumption.} The non-negative parameter $\lambda_{V N}$ controls the rate at which viable tumor cells transit into necrotic cells. Meanwhile,the Heaviside step function, $\mathcal{H}$ is defined as:
\begin{equation}\label{Heaviside}
  \mathcal{H}(x) := \begin{cases}
  1, & x \geq 0, \\
  0, & x < 0.
  \end{cases}
  \end{equation}
   In equation \eqref{phiN}, $\mathcal{H}$ acts as a ``switch'' that regulates necrotic cell formation based on the local nutrient concentration $\phi_\sigma$. {Specifically, when $\phi_\sigma$ drops below the threshold value $\sigma_{V N}$, $\mathcal{H}\left(\sigma_{V N} - \phi_\sigma\right)$ is activated (equals 1), causing an increase in $\phi_N$. This reflects the biological process where insufficient nutrient supply leads to cell death and an expansion of the necrotic region.}

   {Moreover, the necrotic tumor cells volume fraction $\phi_N$ should be less than or equal to the total tumor volume fraction $\phi_T$, and all volume fraction variables $phi_T$, $phi_V$, $phi_N$ are within the range \([0, 1]\).}

\subsection{Critical nutrient dynamics}
The nutrient concentration, denoted as $\phi_\sigma$, includes substances such as oxygen or glucose. Assume the nutrient concentration decreases as it is consumed by viable tumor cells at a positive constant consumption rate $\lambda_\sigma$. For our simulations, we set $\lambda_\sigma = 1.5$. The nutrient concentration $\phi_\sigma=\phi_\sigma(x,t)$ is dictated by a reaction-diffusion equation:
\begin{equation}\label{phisigma}
\partial_{t}\phi_{\sigma} =\frac{M_{\sigma}}{\delta_{\sigma}}\Delta\phi_{\sigma}-\lambda_{\sigma}\phi_{V}\phi_{\sigma},
\end{equation}
where $\delta_\sigma$ and $M_{\sigma}$ are the respective positive constants associated with nutrient diffusion and mobility. {
    It should be noted that the nutrient concentration \(\phi_\sigma\) is non-negative and cannot exceed its initial maximum value \(|\phi_{\sigma}^0|_{\max}\), since nutrients are consumed without external supply.}

\subsection{MDE dynamics}
When the local supply of nutrients $\phi_\sigma$ dips below a certain threshold, tumor cells can transition into a hypoxic state. Then, hypoxic cells secrete enzymes that facilitate cell migration by degrading the ECM. The MDE production rate by viable cells is directly proportional to the nutrient concentrations and the density of ECM, with a rate represented by $\lambda_M^\mathrm{pro}$. The MDE decreases due to a natural decay process at a rate of $\lambda_M^\mathrm{dec}$, and is also affected by the degradation of the ECM, which occurs at a rate of $\lambda_\theta^\mathrm{dec}$. The volume fraction of MDE $\phi_M$ is dictated by the following equation:
\begin{equation}\label{phiM}
\partial_{t}\phi_{M}=M_M\Delta\phi_M-\lambda_M^\mathrm{dec}\phi_M  +\lambda_{M}^{\mathrm{pro}}\phi_{V}\theta\frac{\sigma_{H}}{\sigma_{H}+\phi_{\sigma}}(1-\phi_{M})-\lambda_{\theta}^{\mathrm{dec}}\theta\phi_{M},
\end{equation}
where $M_M$ is a positive constant associated with the mobility. {The volume fraction of MDE, \(\phi_M\), should be constrained within the interval \([0, 1]\).}

\subsection{ECM dynamics}\label{secECM}
The density of the ECM $\theta$, is another variable assumed to be non-diffusive in the model.  Its evolution can be described by a logistic-type equation, reflecting the degradation of ECM caused by specific MDEs. The degradation rate of ECM is controlled by the parameter $\lambda_{\theta}^{\mathrm{deg}}$. The dynamics of $\theta$ are governed by the following equation:
\begin{equation}\label{theta}
\partial_{t}\theta =-\lambda_{\theta}^{\mathrm{deg}}\theta\phi_{M}.
\end{equation}

{Since the model does not account for ECM generation, \(\theta\) is always constrained by \(0 \leq \theta \leq |\theta^0|_{\max}\), meaning that its density cannot exceed the initial maximum value and remains non-negative throughout the degradation process.}

\subsection{The coupled tumor growth system}
{Summarizing the previous sections, the model describes tumor progression by incorporating multiple interacting components, including the total volume fraction of tumor cells \(\phi_T\), the volume fraction of necrotic tumor cells \(\phi_N\), the nutrient concentration \(\phi_\sigma\), the volume fraction of matrix-degrading enzymes (MDEs) \(\phi_M\), and the density of the extracellular matrix (ECM) \(\theta\). } Ultimately, we obtain the following coupled governing system:
\begin{equation}\label{Modelsystem}
\left\{
\begin{aligned}
\partial_t \phi_T &= \operatorname{div} (m_T\left(\phi_V\right)\nabla \mu )-\operatorname{div}J_\alpha\left(\phi_V, \theta\right)+\lambda_T^{\mathrm{pro}} \phi_\sigma \phi_V\left(1-\phi_T\right)-\lambda_T^{\mathrm{apo}} \phi_V , \\
\mu &= \bar{E}\left(4\phi_T^3-6\phi_T^2+2\phi_T\right)-\varepsilon_T^2\Delta\phi_T , \\
\partial_t \phi_N&=\lambda_{V N} \mathcal{H}\left(\sigma_{V N}-\phi_\sigma\right) \phi_V , \\
\partial_{t}\phi_{\sigma}& =\frac{M_{\sigma}}{\delta_{\sigma}}\Delta\phi_{\sigma}-\lambda_{\sigma}\phi_{V}\phi_{\sigma} ,
\\
\partial_{t}\phi_{M}&=M_M\Delta\phi_M-\lambda_M^\mathrm{dec}\phi_M  +\lambda_{M}^{\mathrm{pro}}\phi_{V}\theta\frac{\sigma_{H}}{\sigma_{H}+\phi_{\sigma}}(1-\phi_{M})-\lambda_{\theta}^{\mathrm{dec}}\theta\phi_{M} , \\
\partial_{t}\theta &=-\lambda_{\theta}^{\mathrm{deg}}\theta\phi_{M}.
\end{aligned}
\right.
\end{equation}
with the following initial and homogeneous Neumann boundary conditions, which ensure the steady-state distribution of the model variables at the boundary:
\begin{equation}\label{ICBC}
\begin{aligned}
& \phi_T(0)=\phi_{T,0},\ \phi_\sigma(0)=\phi_{\sigma,0},\ \phi_M(0)=\phi_{M,0},\ \phi_N(0)=\phi_{N,0},\ \theta(0)=\theta_0,
\ \text { in } \Omega,\\
&\quad \frac{\partial \phi_T}{\partial \mathbf{n}}=\frac{\partial \phi_\sigma}{\partial \mathbf{n}}= \frac{\partial \phi_M}{\partial \mathbf{n}}=m_T(\phi_V)\frac{\partial \mu}{\partial \mathbf{n}}-J_{\alpha}(\phi_V,\theta)\cdot\mathbf{n}=0, \quad \text { on } [0, T] \times \partial \Omega,
\end{aligned}
\end{equation}
where $\Omega$ is a bounded domain within the $d$-dimensional real space $\mathbb{R}^d$, $d=\{2,3\}$, with a Lipschitz continuity on the boundary. Moreover, $\mathbf{n}$ represents the outward unit normal vector on the boundary, and \( T > 0 \) represents the final time.

To approximate the Heaviside step function, different smooth functions can be used. Here, we employ a smooth function, as delineated in \cite{liu2022two}, to approximate the Heaviside step function \eqref{Heaviside},

\begin{equation*}
  \mathcal{H}(x) = \frac{1}{2} \Big(1 + \frac{2}{\pi} \arctan \Big(\frac{x}{\epsilon_1} \Big) \Big).
\end{equation*}
In practical simulations, we set $\epsilon_1 = 1e-3$. Alternatively, a sigmoid function can also be used \cite{fritz2019local} for this approximation.

To ensure that $\phi_T$ falls within the range of $[0,1]$, we employ a cut-off operator $\mathscr{C}(\phi_T)=\max(0,\min(1,\phi_T))$. In turn, $\phi_V$ is substituted into the original equation with $\phi_V=\mathscr{C}(\phi_T)-\phi_N$.

Table \ref{param}, shown below, contains all the dimensionless parameter values pertinent to this model system.

\begin{table}[H]
  \centering
   \caption{Parameter values involved in the control system \cite{fritz2019local}}
   \label{param}
  \begin{tabular}{c c @{\hspace{3cm}} c c}
  \hline
  \\[-1.75ex]
      \textbf{Parameter} & \textbf{Value} & \textbf{Parameter} & \textbf{Value}  \\
  \\[-1.75ex]
  \hline
  \\[-2.5ex]
      $\varepsilon_T$ & $0.005$ & $\chi_H$ & $0.001$  \\
      $\bar{E}$ & $0.045$ & $\lambda_T^{\mathrm{pro}}$ & $2$ \\
      $\delta_\sigma$ & $0.01$ & $\lambda_T^{\mathrm{apo}}$ & $0.005$ \\
      $D_\sigma$ & $0.001$ & $\lambda_{V N}$ & $1$ \\
      $M_T$ & $2$ & $\lambda_M^{\mathrm{pro}}$ & $1$ \\
      $M_M$ & $0.1$ & $\lambda_M^{\mathrm{dec}}$ & $1$  \\
      $\sigma_{V N}$ & $0.44$ & $\lambda_\theta^{\mathrm{deg}}$ & $1$  \\
      $\sigma_H$ & $0.6$ & $\lambda_\theta^{\mathrm{dec}}$ & $0.1$ \\
      \\[-2.5ex]
      \hline
  \end{tabular}
\end{table}

\section{The numerical method}

In this section, we propose decoupled and efficient numerical schemes, based on the exponential time differencing (ETD) approach, to solve the coupled tumor growth model \eqref{Modelsystem}.

{Set the time step size as \( \tau = T/N_t \), where \( t_n = n \tau \) for \( n = 0, \ldots, N_t \).} The design of numerical schemes should also present the properties of the maximum bound principle (MBP) for the nutrient concentration and MDE volume fraction, as well as the bound preservation for the ECM density and tumor volume fraction, that is,

\begin{enumerate}[(i).]
  \setlength\itemsep{0em}
  \item For total tumor and necrotic core volume fractions: $0 \leq \phi_N(\cdot, t_n) \leq \phi_T(\cdot, t_n) \leq 1$.
  \item For nutrient concentration: $\phi_\sigma(\cdot, t_n)  \in [0, |\phi_{\sigma}^0|_{\max}]$.
  \item For MDE volume fraction: $\phi_M(\cdot, t_n)  \in [0,1]$.
  \item For ECM density: $\theta(\cdot, t_n)  \in [0, |\theta^0|_{\max}]$.
\end{enumerate}

{
In this section, we use the central finite difference scheme for spatial discretization and choose the ETDRK method and trapezoidal rule for temporal discretization. Additionally, fast computations based on FFT are employed to enhance computational efficiency. After applying a cutoff operator to the updated \(\phi_T\) at each time step, we prove that the ETDRK method preserves the boundedness of \(\phi_\sigma\) and \(\phi_M\). For \(\phi_N\) and \(\theta\), the trapezoidal rule ensures their boundedness under appropriate time step choices (which can be easily satisfied).
}


\subsection{Two-dimensional spatial discretization}\label{space_disc}

For spatial discretization, we apply the central finite difference approximation, using the notation and results for some discrete functions and operators as described in \cite{chen2019positivity,wise2010unconditionally}. A rectangular domain $\Omega$ is considered, defined by $x\in(-1,1)$ and $y\in(-1,1)$, {where $h_{x}=2/m_x$ and $h_{y}=2/n_y$, with $m_x$ and $n_y$ being positive integers. The numerical solutions are denoted as ${\phi_T}_{i, j}(t) \approx {\phi_T}\left(x_i, y_j,t\right)$, ${\phi_\sigma}_{i, j}(t) \approx {\phi_\sigma}\left(x_i, y_j,t\right)$, and ${\phi_M}_{i, j}(t) \approx {\phi_M}\left(x_i, y_j,t\right)$}, applicable for $0 \leq i \leq m_x$ and $0 \leq j \leq n_y$.

Consider the following three sets:
$$
\begin{aligned}
&E_{m}=\{-1+i \cdot h_x \mid i=0, \ldots, m_x\}\\
& C_{m}=\left\{-1+ (i-\frac{1}{2} ) \cdot h_x  \Big\rvert\, i=1, \ldots, m_x\right\}, \\
& C_{\bar{m}}=\left\{-1+ (i-\frac{1}{2} ) \cdot h_x \Big\rvert\, i=0, \ldots, m_x+1\right\} .
\end{aligned}
$$

The points in set $E_{m}$ are called edge-centered points, while the points in sets $C_{m}$ and $C_{\bar{m}}$ are referred as cell-centered points. The points in the set $C_{\bar{m}} \backslash C_{m}$ are the ghost points.  Similar notations could be introduced for the sets $E_{n}$, $C_{n}$, and $C_{\bar{n}}$.

Define the function spaces as follows:
$$
\begin{array}{ll}
\mathcal{C}_{m \times n}=\left\{\phi: C_{m} \times C_{n} \rightarrow \mathbf{R}\right\}, & \mathcal{C}_{\bar{m} \times \bar{n}}=\left\{\phi: C_{\bar{m}} \times C_{\bar{n}} \rightarrow \mathbf{R}\right\}, \\
\mathcal{C}_{\bar{m} \times n}=\left\{\phi: C_{\bar{m}} \times C_{n} \rightarrow \mathbf{R}\right\}, & \mathcal{C}_{m \times \bar{n}}=\left\{\phi: C_{m} \times C_{\bar{n}} \rightarrow \mathbf{R}\right\}, \\
\mathcal{E}_{m \times n}^{\mathrm{ew}}=\left\{f: E_{m} \times C_{n} \rightarrow \mathbf{R}\right\}, & \mathcal{E}_{m \times n}^{\mathrm{ns}}=\left\{f: C_{m} \times E_{n} \rightarrow \mathbf{R}\right\} .
\end{array}
$$

The functions of $\mathcal{C}_{m \times n}, \mathcal{C}_{\bar{m} \times n}, \mathcal{C}_{m \times \bar{n}}$, and $\mathcal{C}_{\bar{m} \times \bar{n}}$ are called the cell-centered functions. These functions are characterized by $\phi_{i, j}:=\phi\left(x_{i}, y_{j}\right)$, where $x_{i}=-1+ (i-\frac{1}{2}) \cdot h_x, y_{j}=-1+ (j-\frac{1}{2} ) \cdot h_y$, with $i$ and $j$ taking integer and half-integer values.

The functions of $\mathcal{E}_{m \times n}^{\mathrm{ew}}$ and $\mathcal{E}_{m \times n}^{\mathrm{ns}}$ are known as east-west and north-south edge-centered functions. The east-west edge-centered functions are denoted by $f_{i+\frac{1}{2}, j}:=f (x_{i+\frac{1}{2}}, y_{j} )$, and the north-south edge-centered functions are denoted by $f_{i, j+\frac{1}{2}}:=f (x_{i}, y_{j+\frac{1}{2}} )$.

The operators $d_{x}: \mathcal{E}_{m \times n}^{\mathrm{ew}} \rightarrow \mathcal{C}_{m \times n}$ and $d_{y}: \mathcal{E}_{m \times n}^{\mathrm{ns}} \rightarrow \mathcal{C}_{m \times n}$ are defined component-wise as edge-to-center difference operators:
$$
d_{x} f_{i, j}=\frac{1}{h_x}\left(f_{i+\frac{1}{2}, j}-f_{i-\frac{1}{2}, j}\right), \quad d_{y} f_{i, j}=\frac{1}{h_y}\left(f_{i, j+\frac{1}{2}}-f_{i, j-\frac{1}{2}}\right), i=1, \ldots, m_x,j=1, \ldots, n_y.
$$
Define the center-to-edge average and difference operators:
$$
A_{x} \phi_{i+\frac{1}{2}, j}=\frac{1}{2}\left(\phi_{i, j}+\phi_{i+1, j}\right), \quad D_{x} \phi_{i+\frac{1}{2}, j}=\frac{1}{h_x}\left(\phi_{i+1, j}-\phi_{i, j}\right),i=0, \ldots, m_x,j=1, \ldots, n_y,
$$
$$
A_{y} \phi_{i, j+\frac{1}{2}}=\frac{1}{2}\left(\phi_{i, j}+\phi_{i, j+1}\right), \quad D_{y} \phi_{i, j+\frac{1}{2}}=\frac{1}{h_y}\left(\phi_{i, j+1}-\phi_{i, j}\right), i=1, \ldots, m_x,j=0, \ldots, n_y ,
$$
where $A_{x}$, $D_{x}$: $\mathcal{C}_{\bar{m} \times n}\rightarrow \mathcal{E}_{m \times n}^{\text {ew }}$, and $A_{y}$, $D_{y}$: $\mathcal{C}_{m \times \bar{n}}\rightarrow \mathcal{E}_{m \times n}^{\text {ns }}$.

The discrete gradient $\nabla_h: \mathcal{C}_{\bar{m} \times \bar{n}} \rightarrow \overrightarrow{\mathcal{E}}_{m \times n}:= \mathcal{E}_{m \times n}^{\mathrm{ew}}\times\mathcal{E}_{m \times n}^{\mathrm{ns}}$ is defined via
$$
\nabla_h \phi_{i, j }= ( D_x \phi_{i+\frac{1}{2}, j}, D_y \phi_{i, j+\frac{1}{2}} ),
$$
and the discrete divergence $\nabla_h \cdot: \overrightarrow{\mathcal{E}}_{m \times n} \rightarrow \mathcal{C}_{m \times n} $ becomes
$$
\nabla_h \cdot f_{i, j }=d_x f_{i, j }^x+d_y f_{i, j }^y,  \quad \mbox{for} \, \, \,
 f=\left(f^x, f^y \right) \in \overrightarrow{\mathcal{E}}_{m \times n} .
$$

The standard 2D discrete Laplacian $\Delta_{h}$, which maps from $\mathcal{C}_{\bar{m} \times \bar{n}}$ to $\mathcal{C}_{m \times n}$, is defined as follows:
\begin{equation}\label{Lap}
\begin{array}{rlr}
&\Delta_{h} \phi_{i, j}  =d_{x}\left(D_{x} \phi\right)_{i, j}+d_{y}\left(D_{y} \phi\right)_{i, j} \\
& =\frac{1}{h_x^{2}}\left(\phi_{i+1, j}+\phi_{i-1, j}-2 \phi_{i, j}\right)+\frac{1}{h_y^{2}}\left(\phi_{i, j+1}+\phi_{i, j-1}-2 \phi_{i, j}\right).
\end{array}
\end{equation}

In the context of the problem of variable mobility, we also have the following discretization:
\begin{equation*}
\begin{aligned}
& \operatorname{div} (M\left(\phi\right) \nabla_h \mu)_{i,j}=d_{x}\left(M\left(A_{x} \phi\right) D_{x} \mu\right)_{i, j}+d_{y}\left(M\left(A_{y} \phi\right) D_{y} \mu\right)_{i, j}.
\end{aligned}
\end{equation*}

{For the numerical simulations, we assume that $m_x=n_y=N$, which implies $h_x=h_y=h$, where $h$ represents the mesh size.}

\subsection{Fully discrete numerical scheme for $\phi_T$}\label{T_full_discret}

We introduce the stabilization constants, denoted as \(\kappa_{T_1}\) and \(\kappa_{T_2}\).

Specifically, we have
\begin{equation}\label{eq_phi_T}
  \partial_t \phi_T+\mathcal{L}_{\kappa_{T}} \phi_T = N_T(\phi_T, \phi_V, \phi_{\sigma}, \theta),
  \end{equation}
\begin{equation*}
\begin{aligned}
\mathcal{L}_{\kappa_{T}}&=-\kappa_{T_1} \Delta+\kappa_{T_2} \Delta^{2},\\
N_T(\phi_T, \phi_V, \phi_{\sigma}, \theta) &= \nabla \cdot(m_T\left(\Phi_V\right) \nabla \mu )-\nabla \cdot \left(\chi\phi_V\nabla\theta\right)\\
&+\lambda_T^{\mathrm{pro}} \phi_\sigma \phi_V \left(1 - \phi_T\right) -\lambda_T^{\mathrm{apo}} \phi_V -\kappa_{T_1} \Delta\phi_T+\kappa_{T_2} \Delta^{2}\phi_T,\\
\mu &= \bar{E}\left(4\phi_T^3-6\phi_T^2+2\phi_T\right)-\varepsilon_T^2\Delta\phi_T.
\end{aligned}
\end{equation*}

After the spatial semi-discretization of the equation \eqref{eq_phi_T} for \(\phi_T\) using the central finite difference method discussed in Sec.~\ref{space_disc}, a system of ordinary differential equations (ODEs) is obtained:
{\begin{equation}\label{kpT}
\frac{d\Phi_T}{dt}+L_{\kappa_{T}, h} \Phi_T = N_T(\Phi_T, \Phi_V, \Phi_{\sigma}, \Theta),
\end{equation}
where \(\Phi_T\) is discretized as a vector over the grid points:
$$
\Phi_T = \left( {\phi_T}_{1,1},\, {\phi_T}_{1,2},\, \ldots,\, {\phi_T}_{1,N},\, {\phi_T}_{2,1},\, {\phi_T}_{2,2},\, \ldots,\, {\phi_T}_{2,N},\,\ldots,\,{\phi_T}_{N,1},\,\ldots,\, {\phi_T}_{N,N} \right)^T.
$$
And \(\Phi_\sigma\), \(\Phi_N\), \(\Phi_V\), \(\Theta\) are discretized in the same way as \(\Phi_T\):
$$
\Phi_\sigma = \left( {\phi_\sigma}_{1,1},\, {\phi_\sigma}_{1,2},\, \ldots,\, {\phi_\sigma}_{N,N} \right)^T, \quad
\Phi_N = \left( {\phi_N}_{1,1},\, {\phi_N}_{1,2},\, \ldots,\, {\phi_N}_{N,N} \right)^T,
$$
$$
\Phi_V = \left( {\phi_V}_{1,1},\, {\phi_V}_{1,2},\, \ldots,\, {\phi_V}_{N,N} \right)^T, \quad
\Theta = \left( {\theta}_{1,1},\, {\theta}_{1,2},\, \ldots,\, {\theta}_{N,N} \right)^T.
$$
The terms \(\Phi_T^3\) and \(\Phi_T^2\) represent the cubic and quadratic terms of the discretized \(\Phi_T\), respectively:
$$
\Phi_T^3 = \left( {\phi_T}_{1,1}^3, \, {\phi_T}_{1,2}^3, \, \dots, \, {\phi_T}_{N,N}^3 \right)^T,\quad \Phi_T^2 = \left( {\phi_T}_{1,1}^2, \, {\phi_T}_{1,2}^2, \, \dots, \, {\phi_T}_{N,N}^2 \right)^T.
$$
\(L_{\kappa_{T}, h}\) denotes the discrete matrix form associated with the spatially discretized operator \(\mathcal{L}_{\kappa_{T}}\), and \(N_T(\Phi_T,\Phi_V, \Phi_{\sigma}, \Theta)\) represents the spatially discretized nonlinear term.
}

{Notice that, to balance the second-order and fourth-order terms in the equation \eqref{eq_phi_T} and to avoid high-frequency oscillations as well as numerical instability arising from the discretization process, we choose the stabilization parameters as \(\kappa_{T_1} = 2 M_T \bar{E}\) and \(\kappa_{T_2} = (C \varepsilon_T)^2\), where \(C\) is a constant. In simulations, we set \(C = 1/8\). }

By applying the {variation of constants method \cite{hartman2002ordinary}} to equation \eqref{kpT}, we then integrate the equation over a single time step from \(t = t_n\) to \(t = t_{n+1}\). The resulting expression turns out to be

\begin{equation}\label{eqintform}
\Phi_T\left(t_{n+1}\right)=\mathrm{e}^{-L_{\kappa_{T}, h}  \tau} \Phi_T\left(t_n\right)+\int_0^\tau \mathrm{e}^{-L_{\kappa_{T}, h}(\tau-s)} N_{T}(\Phi_T(t_n+s),\Phi_V(t_n+s),\Phi_{\sigma}(t_n+s),\Theta(t_n+s)) ds.\quad
\end{equation}

{ The core idea behind the ETD method is to approximate the nonlinear part $N_T$ of the integral in \eqref{eqintform} using interpolation, while accurately evaluating the integral.
}

The fully discrete low-order ETD scheme comes from approximating $N_T(\Phi_T(t_n+s), \Phi_V(t_n+s), \Phi_{\sigma}(t_n+s), \Theta(t_n+s))$ with $N_T(\Phi_T(t_n), \Phi_V(t_n), \Phi_{\sigma}(t_n), \Theta(t_n))$. The detailed form is given by
\begin{equation*}
\Phi_T^{n+1}=\mathrm{e}^{-L_{\kappa_{T}, h}  \tau} \Phi_T^{n}+N_T(\Phi_T^{n}, \Phi_V^{n}, \Phi_{\sigma}^{n}, \Theta^{n})\int_0^\tau \mathrm{e}^{-L_{\kappa_{T}, h} (\tau-s)} ds.
\end{equation*}

Alternatively, the fully discrete ETD Runge-Kutta method is derived by approximating $N_T(\Phi_T(t_n+s), \Phi_V(t_n+s), \Phi_{\sigma}(t_n+s), \Theta(t_n+s))$ with a linear interpolation based on $N_T(\Phi_T(t_n), \Phi_V(t_n), \Phi_{\sigma}(t_n), \Theta(t_n))$ and $N_T(\widetilde\Phi_T^{n+1}, \widetilde\Phi_V^{n+1}, \widetilde\Phi_{\sigma}^{n+1}, \Theta^{n})$, and the form is represented as
\begin{equation*}
\left\{
\begin{aligned}
\widetilde{\Phi}_T^{n+1}  & =\mathrm{e}^{-L_{\kappa_{T}, h}  \tau} \Phi_T^n+N_T(\Phi_T^{n}, \Phi_V^{n}, \Phi_{\sigma}^{n}, \Theta^{n})\int_0^\tau \mathrm{e}^{-L_{\kappa_{T}, h} (\tau-s)} ds, \\
 \Phi_T^{n+1} & = \mathrm{e}^{-L_{\kappa_{T}, h} \tau} \Phi_T^n+\int_0^\tau \mathrm{e}^{-L_{\kappa_{T}, h}(\tau-s)} \Big( (1-\frac{s}{\tau} ) N_T(\Phi_T^{n}, \Phi_V^{n}, \Phi_{\sigma}^{n}, \Theta^{n})
 \\
   &  \qquad \qquad \qquad \qquad
   +\frac{s}{\tau}N_T(\widetilde\Phi_T^{n+1}, \widetilde\Phi_V^{n+1}, \widetilde\Phi_{\sigma}^{n+1}, \Theta^{n}) \Big) ds .
\end{aligned}
\right.
\end{equation*}
 The $\Phi_{\sigma}^{n}$, $\widetilde \Phi_{\sigma}^{n+1}$, $\Phi_{N}^{n}$ and $\Theta^{n}$ will be described in Sec.~\ref{sigma_full_discret}, Sec.~\ref{N_full_discret} and Sec.~\ref{theta_full_discret}.

 In biological modeling, the order parameter $\Phi_T$ should remain within the [0, 1] range, to maintain the physical relevance. Therefore, we apply a cut-off operator to update $\Phi_T^{n+1}$  and $\widetilde\Phi_V^{n+1}=\mathscr{C} (\widetilde\Phi_T^{n+1} )-\Phi_N^{n}$.
 {
  In fact, the solution of the Cahn-Hilliard type evolutionary equation may develop a maximum norm greater than 1 in finite time, as proved in an existing theoretical work~\cite{LiD2017}, a cut-off process turns out to be a useful tool. Therefore, an appearance of a numerical solution of $\Phi_T^{n+1}$ with a maximum norm greater than 1 does not come from the numerical instability; instead, such a numerical behavior is associated with the PDE analysis of this Cahn-Hilliard type equation. On the other hand, since the numerical values of $\widetilde\Phi_T^{n+1}$ usually stay within a slightly deviated range of order $[ - O (\varepsilon), 1 + O (\varepsilon) ]$ in most experiments, as demonstrated in the simulation of Cahn-Hilliard type equation, this approach ensures the numerical solution to meet physical properties and prevents non-physical values in simulations.} This cut-off approach ensures that the numerical schemes for the remaining four variables automatically satisfy the maximum bound principle and preserves the bound.

\subsection{Fully discrete numerical scheme for $\phi_N$}\label{N_full_discret}
After the spatial semi-discretization of the equation \eqref{phiN}, a system of ordinary differential equations (ODEs) based on the discrete grid points is obtained:
{
  \begin{equation}\label{spacephiN}
  \frac{d\Phi_N}{dt}=\lambda_{V N} \mathcal{H}\left(\sigma_{VN}-\Phi_\sigma\right) \odot \Phi_V,
   \end{equation}
   where the symbol \(\odot\) denotes the Hadamard (element-wise) product.
   The trapezoidal rule is used to {construct} the fully discrete scheme for $\Phi_N$:
\begin{equation}\label{TR_phi_N}
  \Phi_{N}^{n+1} = \Phi_{N}^{n} + \frac{\tau}{2} \left( \lambda_{VN}\mathcal{H}(\sigma_{VN}-\Phi_{\sigma}^{n})\odot\left(\Phi_{T}^{n}-\Phi_{N}^{n}\right) + \lambda_{VN}\mathcal{H}(\sigma_{VN}-\widetilde{\Phi}_{\sigma}^{n+1})\odot\left(\Phi_{T}^{n+1}-\Phi_{N}^{n+1}\right) \right),
  \end{equation}}
  where the \(\Phi_{\sigma}^{n}\) and \(\widetilde{\Phi}_{\sigma}^{n+1}\) will be described in Sec.~\ref{sigma_full_discret}.

\begin{theorem}\label{theoremBPN}
  (Discrete bound preservation for $\Phi_N$ of the trapezoidal rule scheme): If the initial values satisfy  \(0\leq\phi_{N_{i,j}}^{0} \leq \phi_{T_{i,j}}^{0}\leq 1\) hold for any \(t_{n+1}\), then the trapezoidal rule scheme \eqref{TR_phi_N} maintains the discrete bound preservation. Specifically, for any time step \(\tau \leq \frac{2}{\lambda_{VN}}\) (which is typically easy to satisfy), the numerical solution \(\phi_{N_{i,j}}^n\) generated by the trapezoidal rule scheme ensures that, for all \(n \geq 0\), we have \(0\leq\phi_{N_{i,j}}^{n} \leq \phi_{T_{i,j}}^{n}\leq 1\) hold for any \(t_{n+1}\).
  \end{theorem}

\begin{proof}
    A simplification of equation \eqref{TR_phi_N} gives us an update rule for $\phi_{N_{i,j}}^{n+1}$ :
\begin{equation*}
\begin{aligned}
(1&+\frac{\tau}{2}\lambda_{VN}\mathcal{H}(\sigma_{VN}-\widetilde\phi_{\sigma_{i,j}}^{n+1}))\phi_{N_{i,j}}^{n+1} = (1-\frac{\tau}{2}\lambda_{VN}\mathcal{H}(\sigma_{VN}-\phi_{\sigma_{i,j}}^{n}))\phi_{N_{i,j}}^{n} \\
&+ \frac{\tau}{2} \left( \lambda_{VN}\mathcal{H}(\sigma_{VN}-\phi_{\sigma_{i,j}}^{n})\phi_{T_{i,j}}^{n} + \lambda_{VN}\mathcal{H}(\sigma_{VN}-\widetilde\phi_{\sigma_{i,j}}^{n+1})\phi_{T_{i,j}}^{n+1} \right) ,
\end{aligned}
\end{equation*}
in which $\phi_{N_{i,j}}^{n}=\phi_N(x_i,y_j,t_n)$.

Given that \(\phi_{N_{i,j}}^n\), \(\phi_{T_{i,j}}^n\), and \(\phi_{T_{i,j}}^{n+1}\) lie in \([0, 1]\),  \(\tau \leq \frac{2}{\lambda_{VN}}\), and using the definition of the \(\mathcal{H}\) function, along with the fact that \(\lambda_{VN}\), \(1 + \frac{\tau}{2} \lambda_{VN} \mathcal{H}(\sigma_{VN} - \widetilde{\phi}_{\sigma_{i,j}}^{n+1})\), and \(1 - \frac{\tau}{2} \lambda_{VN} \mathcal{H}(\sigma_{VN} - \phi_{\sigma_{i,j}}^n)\) are all positive, we conclude that both the right-hand side and the coefficient of the left-hand side are non-negative. Therefore, we have \(\phi_{N_{i,j}}^{n+1} \geq 0\). Moreover, we have

\begin{equation}\label{TR_phi_N3}
\begin{aligned}
&(1+\frac{\tau}{2}\lambda_{VN}\mathcal{H}(\sigma_{VN}-\widetilde\phi_{\sigma_{i,j}}^{n+1}))\phi_{N_{i,j}}^{n+1} \\
&\leq 1-\frac{\tau}{2}\lambda_{VN}\mathcal{H}(\sigma_{VN}-\phi_{\sigma_{i,j}}^{n}) + \frac{\tau}{2} \left( \lambda_{VN}\mathcal{H}(\sigma_{VN}-\phi_{\sigma_{i,j}}^{n}) + \lambda_{VN}\mathcal{H}(\sigma_{VN}-\widetilde\phi_{\sigma_{i,j}}^{n+1}) \right) \\
&= 1+\frac{\tau}{2}\lambda_{VN}\mathcal{H}(\sigma_{VN}-\widetilde\phi_{\sigma_{i,j}}^{n+1}) .
\end{aligned}
\end{equation}
from equation \eqref{TR_phi_N3}, we deduce that \(\phi_{N_{i,j}}^{n+1} \leq 1\). Thus, we conclude that: $0 \leq \phi_{N_{i,j}}^{n+1} \leq 1.$

To prove \(\phi_{N_{i,j}}^{n+1} \leq \phi_{T_{i,j}}^{n+1}\), subtract \(\phi_{T_{i,j}}^{n+1}\) from both sides of the equation \eqref{TR_phi_N}:
\begin{equation}\label{TR_phi_N2}
  \begin{aligned}
  & \left(1 + \frac{\tau}{2} \lambda_{VN} \mathcal{H}(\sigma_{VN}-\widetilde\phi_{\sigma_{i,j}}^{n+1})\right) \left(\phi_{N_{i,j}}^{n+1} - \phi_{T_{i,j}}^{n+1}\right) \\
  &= \left(1- \frac{\tau}{2} \lambda_{VN}\mathcal{H}(\sigma_{VN}-\phi_{\sigma_{i,j}}^{n})\right) \phi_{N_{i,j}}^{n} + \frac{\tau}{2} \lambda_{VN}\mathcal{H}(\sigma_{VN}-\phi_{\sigma_{i,j}}^{n}) \phi_{T_{i,j}}^{n} - \phi_{T_{i,j}}^{n+1}
  \end{aligned}
  \end{equation}
A contradiction argument is applied. Assume that there exists \(t_{n+1}\) such that \(\phi_{N_{i,j}}^{n+1} \leq \phi_{T_{i,j}}^{n+1}\). By the definition of the Heaviside function and the non-negativity of \(\tau\) and \(\lambda_{VN}\), the coefficient of \(\phi_{N_{i,j}}^{n+1} - \phi_{T_{i,j}}^{n+1}\) on the left-hand side is always positive. Therefore, the left-hand side of \eqref{TR_phi_N2} is positive. Meanwhile, since \(\phi_{N_{i,j}}^n\), \(\phi_{T_{i,j}}^n\), and \(\phi_{T_{i,j}}^{n+1}\) are all within \([0,1]\), the minimum value of the right-hand side could reach \(-1\). This contradicts the fact that the left-hand side is positive. Therefore, the assumption \(\phi_{N_{i,j}}^{n+1} > \phi_{T_{i,j}}^{n+1}\) has to be false. Therefore, \(0\leq\phi_{N_{i,j}}^{n+1} \leq \phi_{T_{i,j}}^{n+1}\leq 1\) hold for any \(t_{n+1}\).  Furthermore, since \(\phi_{V_{i,j}}^{n+1} = \phi_{T_{i,j}}^{n+1} - \phi_{N_{i,j}}^{n+1}\), we also have \(\phi_{V_{i,j}}^{n+1} \in [0, 1]\). This completes the proof.
\end{proof}


\subsection{Fully discrete numerical scheme for $\phi_{\sigma}$}\label{sigma_discret}

The numerical design is based on the MBP framework for semi-linear elliptic equations, as established by Du et al. \cite{du2021maximum}. By adhering to a specific set of conditions, we are able to apply the existing analysis. This allows us to easily derive the MBP and bound preserving properties. To facilitate the application of this framework, we first adopt the affine transformation, {$\zeta(\gamma): \mathbb{R}^{d}\rightarrow\mathbb{R}^{d}\  (d\leq 3)$}, defined by the following equation:
\begin{equation*}
\zeta(\gamma)=\frac{1}{2}\gamma+\frac{1}{2},  \quad\gamma\in\mathbb{R}^{d}.
\end{equation*}
Given $\phi_\sigma$, by setting $\psi_{\sigma}=\zeta^{-1}(\frac{1}{|\phi_{\sigma}^0|_{\max}}\phi_{\sigma})$, we can obtain an equivalent new equation for $\psi_{\sigma}$:
\begin{equation*}
\partial_t \psi_{\sigma}=\frac{M_{\sigma}}{\delta_\sigma} \Delta \psi_\sigma-\lambda_{\sigma} \phi_V (\psi_\sigma+1).
\end{equation*}
In turn, this formulation transforms the problem of ensuring $\phi_\sigma\in[0,|\phi_{\sigma}^0|_{\max}]$ into the equivalent problem of ensuring $\psi_\sigma\in[-1,1]$.

Let us introduce a stabilization constant $\kappa_{\sigma}>0$. Then the equation takes the form
\begin{equation}\label{kappaeqsigma}
\begin{aligned}
&\partial_t \psi_\sigma+\mathcal{L}_{\kappa_{\sigma}} \psi_\sigma=\mathcal{N}_{\sigma}\psi_\sigma,\\
&\mathcal{L}_{\kappa_{\sigma}}=-\frac{M_{\sigma}}{\delta_\sigma}\Delta+\kappa_{\sigma}\mathcal{I} , \\
&\mathcal{N}_{\sigma}\psi_{\sigma}=\kappa_{\sigma}\psi_{\sigma}+f_{0,\sigma}(\phi_V,\psi_{\sigma}),  \\
&f_{0,\sigma}(\phi_V,\psi_{\sigma})=-\lambda_{\sigma} \phi_V (\psi_\sigma+1) .
\end{aligned}
\end{equation}
where $\mathcal{I}$ is the identity mapping, with the initial and homogeneous Neumann boundary conditions:
\begin{equation*}
\psi_\sigma(0)=2\phi_{\sigma,0}-1,\text { in } \Omega,\qquad \frac{\partial \psi_\sigma}{\partial \mathbf{n}}=0, \text { on } [0, T] \times \partial \Omega,
\end{equation*}
In addition, the stabilization constant $\kappa_{\sigma}$ is selected as
\begin{equation}\label{kappasigma}
\kappa_{\sigma} \geq \max _{|\xi| \leq 1}\left|f_{0,\sigma}^{\prime}(\xi)\right|=\max _{|\xi| \leq 1}|-\lambda_{\sigma} \phi_V|=\lambda_{\sigma}.
\end{equation}

For the nonlinear term $f_{0,\sigma}$, notice that $\phi_V=\mathscr{C}(\phi_T)-\phi_N$ is within the interval $[0,1]$ and $\lambda_\sigma$ is nonnegative. For any $\beta_\sigma\geq 1$, the following inequality holds:
\begin{equation*}
f_{0,\sigma}(\phi_V,\beta_\sigma)=-\lambda_\sigma \phi_V(\beta_\sigma+1)\leq 0\leq  f_{0,\sigma}(\phi_V,-\beta_\sigma)=-\lambda_\sigma \phi_V(-\beta_\sigma+1).
\end{equation*}

Then we have the following result.
\begin{lemma}\label{lemmanonlinearsigma}
 Under  the requirement  \eqref{kappasigma}, when $\phi_V\in[0,1]$ for any $\xi \in[-\beta_\sigma, \beta_\sigma]$, it holds that
\begin{equation*}
\left|\mathcal{N}_{\sigma}[\xi]\right| \leq \kappa_{\sigma} \beta_\sigma.
\end{equation*}
\end{lemma}
In this study, we set $\beta_\sigma=1$.

\subsubsection{The fully discrete scheme}\label{sigma_full_discret}

Applying the central finite difference method to \eqref{kappaeqsigma} yields a system of ordinary differential equations (ODEs):
{
\begin{equation*}
  \frac{d\Psi_\sigma}{dt}+L_{\kappa_{\sigma,h}} \Psi_\sigma =N_{\sigma}(\Phi_V, \Psi_\sigma),
  \end{equation*}
where \(L_{\kappa_{\sigma}, h}\) represents the discrete matrix form of the spatially discretized operator \(\mathcal{L}_{\kappa_{\sigma}}\), and \(N_{\sigma}(\Phi_V, \Psi_\sigma)\) denotes the spatially discretized nonlinear term.
}

By employing the ETD schemes provided in Sec.~\ref{T_full_discret}, the fully discrete low-order ETD scheme is derived:

{
  \begin{equation}\label{sigmaETD1}
  \Psi_{\sigma}^{n+1}=\mathrm{e}^{-L_{\kappa_{\sigma}, h}  \tau} \Psi_{\sigma}^{n}+N_{\sigma}(\Phi_V^n, \Psi_{\sigma}^n)\int_0^\tau \mathrm{e}^{-L_{\kappa_{\sigma},  h} (\tau-s)} ds.
  \end{equation}}

Moreover, the fully discrete ETD Runge Kutta scheme takes the form of

{\begin{equation}\label{sigmaETDRK2}
  \left\{
  \begin{aligned}
  \quad \widetilde{\Psi}_{\sigma}^{n+1} & = \mathrm{e}^{-L_{\kappa_{\sigma}, h}  \tau} \Psi_{\sigma}^n+N_{\sigma}(\Phi_V^n, \Psi_{\sigma}^n)\int_0^\tau \mathrm{e}^{-L_{\kappa_{\sigma}, h} (\tau-s)} ds, \\
  \quad \Psi_{\sigma}^{n+1} & =  \mathrm{e}^{-L_{\kappa_{\sigma}, h} \tau} \Psi_{\sigma}^n+\int_0^\tau \mathrm{e}^{-L_{\kappa_{\sigma}, h}(\tau-s)}\left(\left(1-\frac{s}{\tau}\right) N_{\sigma}(\Phi_V^n, \Psi_{\sigma}^n)+\frac{s}{\tau}N_{\sigma}(\widetilde{\Phi}_V^{n+1}, \widetilde{\Psi}_{\sigma}^{n+1})\right) ds.
  \end{aligned}
  \right.
  \end{equation}}

Hence, the numerical solutions of equation \eqref{phisigma} can be deduced from the subsequent relationships:

{\begin{equation}\label{relationsigma}
  \widetilde{\Phi}_{\sigma}^{n+1}=|\Phi_{\sigma}^0|_{\max}\frac{\widetilde{\Psi}_{\sigma}^{n+1}+1}{2},  \quad \Phi_{\sigma}^n=|\Phi_{\sigma}^0|_{\max}\frac{\Psi_{\sigma}^n+1}{2},
  \end{equation}}

\subsubsection{Discrete MBP for $\psi_\sigma$}

In order to prove the discrete MBP property, we need the following lemma.

\begin{lemma}[\cite{du2021maximum}]\label{lemmaexp}
  The Laplacian operator $\Delta$ generates a contraction semigroup $\left\{\mathrm{e}^{t\Delta}\right\}_{t \geq 0}$ on $C(\bar{\Omega})$, where $\bar{\Omega} = \Omega \cup \partial \Omega$, and the functions $u$ in $C(\bar{\Omega})$ are continuous on the closed domain and {satisfy the homogeneous Neumann boundary condition.} For any $\alpha \geq 0$, it holds that
  $$
  \left\|\mathrm{e}^{t(\Delta-\alpha)} u\right\| \leq \mathrm{e}^{-\alpha t}\|u\|, \quad t \geq 0,
  $$
  \end{lemma}
  As indicated in \cite{du2021maximum}, the spatially discrete approximation to the Laplacian \(\Delta\), namely $\Delta_h$,  satisfies the conditions of Lemma \ref{lemmaexp}. Therefore, based on Lemma \ref{lemmaexp}, we can obtain the following discrete MBP theorem.

\begin{theorem}\label{theoremMBPsigma}
(Discrete MBP for \(\Psi_\sigma\) of ETD schemes): If $\|\psi_{\sigma}(\cdot,t_0) \|_{\infty}\leq 1$ and \(\kappa_{\sigma}\) satisfies \eqref{kappasigma}, then the ETD scheme \eqref{sigmaETD1} and the ETDRK scheme \eqref{sigmaETDRK2} unconditionally preserve the discrete MBP. In other words, for any time step size \(\tau>0\), the solutions satisfy:
 $\|\Psi_{\sigma}^n\|_{\infty} \leq 1$, for any $n \geq 0$.
\end{theorem}

\begin{proof}
Since $\|\psi_{\sigma}(\cdot,t_0) \|_{\infty}\leq 1$, it suffices to prove the theorem for the case $\|\Psi_{\sigma}^n\|_{\infty} \leq 1$ and then deduce  $\|\Psi_{\sigma}^{n+1}\|_{\infty} \leq 1$ for any $n$.

For the ETD scheme \eqref{sigmaETD1}, an application of Lemmas \ref{lemmanonlinearsigma} and \ref{lemmaexp}, as well as the fact that $\|\Psi_{\sigma}^n\|_{\infty} \leq 1$, implies that
{
\begin{equation}\label{thetd1sigma}
  \begin{aligned}
    \|\Psi_{\sigma}^{n+1}\|_{\infty} & \leq \|\mathrm{e}^{-L_{\kappa_{\sigma}, h} \tau}\|_{\infty} \|\Psi_{\sigma}^n\|_{\infty} + \int_0^\tau \|\mathrm{e}^{-L_{\kappa_{\sigma}, h} (\tau-s)}\|_{\infty} \|N_{\sigma}(\Phi_V^n, \Psi_{\sigma}^n)\|_{\infty} ds, \\
    & \leq \mathrm{e}^{-\kappa_{\sigma} \tau} + \int_0^\tau \mathrm{e}^{-\kappa_{\sigma}(\tau-s)} \kappa_{\sigma} ds =\mathrm{e}^{-\kappa_{\sigma} \tau}+ \frac{1-\mathrm{e}^{-\kappa_{\sigma} \tau}}{\kappa_{\sigma}}\kappa_{\sigma}=1.
    \end{aligned}
  \end{equation}
}

Similarly, for the ETDRK scheme \eqref{sigmaETDRK2}, an application of the two lemmas, combined with the inequalities that $\|\Psi_{\sigma}^n \|_{\infty} \leq 1$, $\|\widetilde\Psi_{\sigma}^{n+1} \|_{\infty} \le 1$ (based on \eqref{thetd1sigma}), gives

{$$
\begin{aligned}
\left\|\Psi_{\sigma}^{n+1}\right\|_{\infty} & \leq\left\|\mathrm{e}^{-L_{\kappa_{\sigma}, h}  \tau} \right\|_{\infty} \left\|\Psi_{\sigma}^n\right\|_{\infty}\\
&\quad+\int_0^\tau \|\mathrm{e}^{-L_{\kappa_{\sigma}, h} (\tau-s)} \|_{\infty} \Big\| (1-\frac{s}{\tau} ) N_{\sigma}(\Phi_V^n, \Psi_{\sigma}^n)+\frac{s}{\tau}N_{\sigma}(\widetilde\Phi_V^{n+1}, \widetilde\Psi_{\sigma}^{n+1}) \Big\|_{\infty} ds, \\
& \leq \mathrm{e}^{-\kappa_{\sigma} \tau}+\int_0^\tau \mathrm{e}^{- \kappa_{\sigma}(\tau-s)} \Big( (1-\frac{s}{\tau} ) \kappa_{\sigma}+\frac{s}{\tau}\kappa_{\sigma} \Big) ds,\\
& =\mathrm{e}^{-\kappa_{\sigma} \tau}+ \frac{1-\mathrm{e}^{-\kappa_{\sigma} \tau}}{\kappa_{\sigma}}\kappa_{\sigma}=1.
\end{aligned}
$$}
This completes the proof.
\end{proof}

Thanks to the \eqref{relationsigma} and Theorem \ref{theoremMBPsigma}, it is obvious that if $\phi_{\sigma}(\cdot,t_0) \in[0,|\phi_{\sigma}^0|_{\max}]$ (or equivalently $\|\psi_{\sigma}(\cdot,t_0)\|_{\infty} \leq 1)$, then the numerical solution satisfies $\phi_{\sigma}(\cdot,t)\in[0,|\phi_{\sigma}^0|_{\max}]$ for all $t > 0$.

{
  \begin{remark}\label{remarkMBPsigma}
    Note that the introduction of the stabilization constant  \(\kappa_\sigma\) is to ensure the validity of the maximum bound principle (MBP). When \(\kappa_\sigma\) satisfies the conditions in Theorem \ref{theoremMBPsigma}, the MBP guarantees the boundedness of the numerical solution. As a measure of numerical stability, the MBP ensures that the solution does not blow up during the computation, which is particularly important when large time steps and large-scale computations are used.
  \end{remark}
      }
\subsection{Fully discrete numerical scheme for $\phi_M$}

Similar to the nutrient concentration equation outlined in Sec.~\ref{sigma_discret}, by setting $\psi_{M}=\zeta^{-1}(\phi_{M})$, in which $\phi_M$ is the volume fraction of MDE, we obtain an equivalent equation for $\psi_M$:
\begin{equation*}
\partial_t \psi_{M}=M_M \Delta \psi_M -\lambda_M^{\mathrm{dec}} (\psi_M+1)+\lambda_M^{\mathrm{pro}} \phi_V \theta \frac{\sigma_H}{\sigma_H+\phi_\sigma}\left(1-\psi_M\right)-\lambda_\theta^{\mathrm{dec}} \theta (\psi_M+1).
\end{equation*}
This transforms the problem of ensuring $\phi_M\in[0,1]$ into an equivalent problem of ensuring $\psi_M\in[-1,1]$.

An introduction of a stabilization constant $\kappa_{M}>0$ gives
\begin{equation}\label{kappaeqM}
\partial_t \psi_M+\mathcal{L}_{\kappa_{M}} \psi_M=\mathcal{N}_{M}\psi_M,
\end{equation}
where
\begin{equation*}
\begin{aligned}
&\mathcal{L}_{\kappa_{M}}=-M_M\Delta+\kappa_{M}\mathcal{I},
 \quad \mbox{$\mathcal{I}$ is the identity mapping} , \\
&\mathcal{N}_{M}\psi_M=\kappa_{M}\psi_M+f_{0,M}(\phi_V,\psi_{M},\phi_{\sigma},\theta),\\
&f_{0,M}(\phi_V,\psi_{M},\phi_{\sigma},\theta)=-\lambda_M^{\mathrm{dec}} (\psi_M+1)+\lambda_M^{\mathrm{pro}} \phi_V \theta \frac{\sigma_H}{\sigma_H+\phi_\sigma}\left(1-\psi_M\right)-\lambda_\theta^{\mathrm{dec}} \theta (\psi_M+1) .
\end{aligned}
\end{equation*}
with the initial and homogeneous Neumann boundary conditions:
\begin{equation*}
\psi_M(0)=2\phi_{M,0}-1, \text { in } \Omega,\qquad \frac{\partial \psi_M}{\partial \mathbf{n}}=0, \text { on } [0, T] \times \partial \Omega,
\end{equation*}

The stabilization constant $\kappa_{M}$ should be chosen as follows:
\begin{equation}\label{kappaM}
\begin{aligned}
\kappa_{M} \geq \max _{|\xi| \leq 1}\left|f_{0,M}^{\prime}(\xi)\right|= & \max _{|\xi| \leq 1}|-\lambda_M^{\mathrm{dec}} -\lambda_M^{\mathrm{pro}} \phi_V \theta \frac{\sigma_H}{\sigma_H+\phi_\sigma}-\lambda_\theta^{\mathrm{dec}} \theta|
\\
  = & \lambda_M^{\mathrm{dec}} +\lambda_M^{\mathrm{pro}}|\theta|_{\max}+\lambda_\theta^{\mathrm{dec}}|\theta|_{\max} .
\end{aligned}
\end{equation}

{The choice of \(\kappa_M\) is similar to \(\kappa_\sigma\), as described in \(\text{Remark \ref{remarkMBPsigma}}\).}

Regarding the nonlinear term $f_{0,M}$, it is noticed that if the initial condition $\theta_0 >0$ and $\phi_M > 0$, then $\theta >0$. A contradiction argument is applied to prove this fact. Suppose there exists $t^*$ such that $\phi_M(t^*) > 0$ and $\theta(t^*) \leq 0$. By the evolutionary equation \eqref{theta}, the derivative of $\theta$ is also greater than zero, i.e., $\theta$ is increasing. Therefore, we see that $\theta(t_0) \leq \theta(t_1) \leq ... \leq \theta(t^*) \leq 0$, which contradicts the fact that $\theta_0 \in [0,1]$. Hence, the assumption becomes invalid, so that $\theta >0$ whenever $\phi_M > 0$. Given $\phi_V,\phi_{\sigma}\in[0,1]$, the parameters $\lambda_M^{\mathrm{dec}}$, $\lambda_M^{\mathrm{pro}}$, $\sigma_H$, and $\lambda_\theta^{\mathrm{dec}}$ are all non-negative. As a result, for any $\beta_M$ such that $\beta_M\geq 1$, the following inequality holds:
\begin{equation*}
f_{0,M}(\phi_V,\beta_M,\phi_{\sigma},\theta)\leq 0 \leq  f_{0,M}(\phi_V,-\beta_M,\phi_{\sigma},\theta).
\end{equation*}

Then we have the following lemma.
\begin{lemma}\label{lemmanonlinearM}
 Under  the requirement  \eqref{kappaM}, for any $\xi \in[-\beta_M, \beta_M]$, if $\phi_V$, $\phi_\sigma \in [0, 1]$, it holds that
\begin{equation*}
\left|\mathcal{N}_{M}[\xi]\right| \leq \kappa_{M} \beta_M.
\end{equation*}
\end{lemma}
In this study, we set $\beta_M=1$.

\subsubsection{Fully discrete scheme}
Applying the central finite difference scheme in space to \eqref{kappaeqM} results in the following system of ordinary differential equations (ODEs):
{
\begin{equation*}
  \frac{d\Psi_M}{dt} +L_{\kappa_{M},h} \Psi_M=N_{M}(\Phi_V, \Psi_M, \Phi_\sigma, \Theta),
\end{equation*}
where $
\Psi_M = \left( {\psi_M}_{1,1},\, {\psi_M}_{1,2},\, \ldots,\,{\psi_M}_{N,N} \right)^T.
$
Here, \(L_{\kappa_{M},h}\) denotes the discrete matrix form associated with the spatially discretized operator \(\mathcal{L}_{\kappa_{M}}\), and \(N_{M}(\Phi_V, \Psi_M, \Phi_{\sigma}, \Theta)\) represents the spatially discretized nonlinear term.
}

Then the fully discrete low-order ETD scheme is obtained:
{\begin{equation}\label{METD1}
  \Psi_{M}^{n+1}=\mathrm{e}^{-L_{\kappa_{M},h}  \tau} \Psi_{M}^{n}+N_{M}(\Phi_V^n, \Psi_M^n, \Phi_\sigma^n, \Theta^n)\int_0^\tau \mathrm{e}^{-L_{\kappa_{M},  h} (\tau-s)} ds.
  \end{equation}}

And the fully discrete ETD Runge Kutta scheme is:
{\begin{equation}\label{METDRK2}
  \left\{
  \begin{aligned}
  \quad \widetilde{\Psi}_{M}^{n+1} = & \mathrm{e}^{-L_{\kappa_{M},h}  \tau} \Psi_M^n+N_{M}(\Phi_V^n, \Psi_M^n, \Phi_\sigma^n, \Theta^n)\int_0^\tau \mathrm{e}^{-L_{\kappa_{M},h} (\tau-s)} ds, \\
  \quad \Psi_{M}^{n+1} = & \mathrm{e}^{-L_{\kappa_{M},h} \tau} \Psi_M^n+\int_0^\tau \mathrm{e}^{-L_{\kappa_{M},h}(\tau-s)} \Big( (1-\frac{s}{\tau} ) N_{M}(\Phi_V^n, \Psi_M^n, \Phi_\sigma^n, \Theta^n) \\
    & \qquad \qquad \qquad +\frac{s}{\tau}N_{M}(\widetilde{\Phi}_V^{n+1}, \widetilde{\Psi}_M^{n+1}, \widetilde{\Phi}_\sigma^{n+1}, \Theta^{n}) \Big) ds,
  \end{aligned}
  \right.
  \end{equation}}

Therefore, the numerical solution to the equation \eqref{phiM} could be derived from the following relationships:
{\begin{equation}\label{relationM}
  \widetilde{\Phi}_{M}^{n+1}=\frac{\widetilde{\Psi}_{M}^{n+1}+1}{2},  \quad \Phi_M^n=\frac{\Psi_M^n+1}{2},
  \end{equation}}

\subsubsection{Discrete MBP for $\psi_M$}

\begin{theorem}\label{theoremMBPM}
(Discrete MBP for $\psi_M$ of the ETD schemes): If \(\|\psi_M(\cdot, t_0)\|_{\infty} \leq 1\) and $\kappa_{M}$ satisfies \eqref{kappaM}, then both the ETD scheme \eqref{METD1} and the ETDRK scheme \eqref{METDRK2} unconditionally uphold the discrete MBP. Specifically, for any time step $\tau>0$, the solution to the ETD scheme \eqref{METD1} and the ETDRK scheme \eqref{METDRK2} satisfy $\|\Psi_M^n\|_\infty \leq 1$ for all $n \geq 0$.
\end{theorem}

\begin{proof}
Analogous to the proof of Theorem \ref{theoremMBPsigma}, we merely need to demonstrate that if $\|\Psi_M^{n}\|_\infty \leq 1$ , it can be inferred that  $\|\Psi_M^{n+1}\|_\infty \leq 1$ for any $n$.

For the ETD scheme \eqref{METD1}, an application of Lemmas \ref{lemmaexp} and \ref{lemmanonlinearM}, combined with the fact that ${\phi_V}_{i,j}^n \in [0,1]$  and $\|\Psi_{M}^n \|_\infty \le 1$, indicates that
{
\begin{equation}\label{thetd1M}
\begin{aligned}
\|\Psi_M^{n+1}\|_\infty & \leq \|\mathrm{e}^{-L_{\kappa_{M},h} \tau}\|_\infty \|\Psi_M^n\|_\infty + \int_0^\tau \|\mathrm{e}^{-L_{\kappa_{M},h} (\tau-s)}\|_\infty \|N_M(\Phi_V^n, \Psi_M^n, \Phi_\sigma^n, \Theta^n)\|_\infty ds, \\
& \leq \mathrm{e}^{-\kappa_M \tau} + \int_0^\tau \mathrm{e}^{-\kappa_M(\tau-s)} \kappa_M ds = \mathrm{e}^{-\kappa_{M} \tau}+ \frac{1-\mathrm{e}^{-\kappa_{M} \tau}}{\kappa_{M}}\kappa_{M}=1.
\end{aligned}
\end{equation}}
 For the ETDRK scheme \eqref{METDRK2}, since $\|\Psi_{M}^n \|_\infty \leq 1$, and $\|{\widetilde\Psi_{M}}^{n+1} \|_\infty \leq 1$ (based on \eqref{thetd1M}), we have
 {\[
  \begin{aligned}
    \|\Psi_{M}^{n+1} \|_\infty \le & \|\mathrm{e}^{-L_{\kappa_{M},h}  \tau} \|_\infty \cdot \|\Psi_{M}^n \|_\infty
      +\int_0^\tau \, \|\mathrm{e}^{-L_{\kappa_{M},h} (\tau-s)} \|_\infty
    \\  \qquad \qquad  &
    \cdot \Big\| (1-\frac{s}{\tau} ) N_{M}(\Phi_V^n, \Psi_{M}^n,\Phi_{\sigma}^n,\Theta^n)+\frac{s}{\tau}N_{M}(\widetilde\Phi_V^{n+1}, \widetilde\Psi_{M}^{n+1},\widetilde\Phi_{\sigma}^{n+1},\Theta^{n}) \Big\|_\infty ds, \\
   \le &  \mathrm{e}^{-\kappa_{M} \tau}+\int_0^\tau \mathrm{e}^{- \kappa_{M}(\tau-s)} \Big( (1-\frac{s}{\tau} ) \kappa_{M}+\frac{s}{\tau}\kappa_{M} \Big) ds,\\
   = & \mathrm{e}^{-\kappa_{M} \tau}+ \frac{1-\mathrm{e}^{-\kappa_{M} \tau}}{\kappa_{M}}\kappa_{M}=1.
   \end{aligned}
\]}

This completes the proof.
\end{proof}

Due to the relationships defined in \eqref{relationM}, it is obvious that if \(\phi_{M}(\cdot,t_0) \in [0,1]\) (or \(\|\psi_{M}(\cdot,t_0)\|_{\infty} \leq 1\)), then it follows that \(\phi_{M}(\cdot,t) \in [0,1]\) for all \(t\).

\subsection{Fully discrete numerical scheme for $\theta$}\label{theta_full_discret}

Applying the central finite difference scheme in space to \eqref{theta}, resulting in the following system of ordinary differential equations (ODEs):
{\begin{equation*}
  \frac{d\Theta}{dt} = -\lambda_{\theta}^{\mathrm{deg}} \Theta \odot \Phi_M,
  \end{equation*}}

To discretize this equation, we adopt the trapezoidal rule, yielding the following numerical scheme:

{\begin{equation}\label{TR_theta}
  \Theta^{n+1} = \Theta^n + \frac{\tau}{2} \left( -\lambda_\theta^{\mathrm{deg}} \Phi_M^n \odot \Theta^n - \lambda_\theta^{\mathrm{deg}} \Phi_M^{n+1} \odot \Theta^{n+1} \right).
  \end{equation}}

\begin{theorem}\label{theoremBPtheta}
(Discrete bound preservation for $\theta$ of the Trapezoidal rule scheme): If $\theta_{i,j}^0 \in \left[ 0, \, \theta^0_{\max} \right], \forall i,j$, then the trapezoidal rule scheme \eqref{TR_theta} preserves the discrete bound. That is, for any time step size $\tau\leq\frac{2}{\lambda_{\theta}^{deg}}$ (which is easy to satisfy), then the trapezoidal rule numerical solution satisfies: $\theta_{i,j}^n \in \left[ 0, \, \theta^0_{\max} \right], \forall i,j,$, for any $n \geq 0$.
\end{theorem}

\begin{proof}

If we can prove that if \(   \theta_{i,j}^n \in \left[0, \theta^0_{\max}\right]\), then \(   \theta_{i,j}^{n+1} \in \left[0, \theta^0_{\max}\right]\), the theorem will be proven. First, simplifying equation \eqref{TR_theta} yields the update formula for \(\theta^{n+1}_{i,j}\):
\begin{equation*}
\theta^{n+1}_{i,j} = \Big(\frac{1 - \frac{\tau}{2} \lambda_\theta^{deg} \phi_{M_{i,j}}^{n}}{1 + \frac{\tau}{2} \lambda_\theta^{deg} \phi_{M_{i,j}}^{n+1}} \Big) \theta^n_{i,j} , \quad \mbox{with} \, \, \,
\theta^n_{i,j} = \theta(x_i, y_j,t_n) .
\end{equation*}
where \(\phi_{M_{i,j}}^n, \phi_{M_{i,j}}^{n+1} \in [0,1]\) by Theorem \ref{theoremMBPM}, and \(\lambda_\theta^{\deg} > 0\), we see that if $ \tau \le \frac{2}{\lambda_\theta^{deg}}$, then
\begin{equation*}
0 \leq \Big(\frac{1 - \frac{\tau}{2} \lambda_\theta^{deg} \phi_{M_{i,j}}^{n}}{1 + \frac{\tau}{2} \lambda_\theta^{deg} \phi_{M_{i,j}}^{n+1}} \Big) \leq 1.
\end{equation*}

Combining this with the inductive hypothesis \(   \theta_{i,j}^n \in \left[0, \theta^0_{\max}\right]\), we obtain:
\[
  0\leq \theta_{i,j}^{n+1} \leq \theta_{i,j}^n \leq \ldots \leq \theta^0_{\max}.
\]

Here the parameter $\lambda_\theta^{deg}$ is set to 1 in the practical computation, thus it is straightforward to make $\tau\leq 2$ satisfied. The proof is complete.
\end{proof}

\subsection{Fast implementation of the ETD schemes}
In the following, we provide an efficient implementation of the ETD schemes utilizing fast algorithms based on the Fast Fourier Transform (FFT).

{
Define the $\{\Upsilon_i\}_{i=0,1,2}$ functions:
\begin{equation*}
\Upsilon_0(L):=\mathrm{e}^{-L},  \quad\Upsilon_1(L):=\frac{1-\mathrm{e}^{-L}}{L}, \quad\Upsilon_2(L):=\frac{\mathrm{e}^{-L}-1+L}{L^2}, \quad L \neq 0.
\end{equation*}}
Subsequently, the numerical scheme could be rewritten as the following equivalent system:

(Fast algorithm of the ETDRK scheme)
\begin{equation}\label{fastETDRK}
\left\{
\begin{aligned}
\quad \widetilde\Phi_T^{n+1}& = \Upsilon_0\left(L_{\kappa_{T},h} \tau\right) \Phi_T^{n}+\tau \Upsilon_1\left(L_{\kappa_{T},h} \tau\right)N_{T}(\Phi_T^{n},\Phi_V^n,\Phi_\sigma^n,\Theta^n),\\
\quad \widetilde{\Psi}^{n+1}_\sigma& = \Upsilon_0\left(L_{\kappa_{\sigma},h} \tau\right) \Psi^n_\sigma+\tau \Upsilon_1\left(L_{\kappa_{\sigma},h} \tau\right)N_{\sigma}(\Phi_V^n,\Psi_\sigma^n),\\
\quad \widetilde{\Psi}^{n+1}_M& = \Upsilon_0\left(L_{\kappa_{M},h} \tau\right) \Psi_M^n+\tau \Upsilon_1\left(L_{\kappa_{M},h} \tau\right)N_{M}(\Phi_V^n,\Psi_M^n,\Phi_\sigma^n,\Theta^n),\\
\widetilde\Phi_\sigma^{n+1}&=(\widetilde\Psi_\sigma^{n+1}+1)/2,\\
\widetilde\Phi_V^{n+1}&=\mathscr{C}\left(\widetilde\Phi_T^{n+1}\right)-\Phi_N^{n},\\
\hat\Phi_T^{n+1}&=\mathscr{C}\left(\widetilde\Phi_T^{n+1}\right),\\
\quad \bar\Phi_T^{n+1} & = \hat\Phi_T^{n+1}+\tau \Upsilon_2\left(L_{\kappa_{T},h}\tau\right)\left(N_{T}(\hat\Phi_T^{n+1},\widetilde{\Phi}^{n+1}_V,\widetilde{\Phi}_\sigma^{n+1},\Theta^{n})-N_{T}(\Phi_T^{n},\Phi_V^n,\Phi_\sigma^n,\Theta^n)\right),\\
\quad \Psi_\sigma^{n+1} & = \widetilde\Psi_\sigma^{n+1}+\tau \Upsilon_2\left(L_{\kappa_{\sigma},h}\tau\right)\left(N_{\sigma}(\widetilde\Phi_V^{n+1},\widetilde\Psi_\sigma^{n+1})-N_{\sigma}(\Phi_V^n,\Psi_\sigma^n)\right),\\
\quad \Psi^{n+1}_M & = \widetilde{\Psi}^{n+1}_M+\tau \Upsilon_2\left(L_{\kappa_{M},h}\tau\right)\left(N_{M}(\widetilde{\Phi}^{n+1}_V,\widetilde\Psi_M^{n+1},\widetilde\Phi_\sigma^{n+1},\Theta^{n})-N_{M}(\Phi_V^n,\Psi_M^n,\Phi_\sigma^n,\Theta^n)\right),\\
\Phi^{n+1}_\sigma&=(\Psi_\sigma^{n+1}+1)/2,\\
\Phi^{n+1}_M&=(\Psi^{n+1}_M+1)/2,\\
\Phi_T^{n+1}&=\mathscr{C}\left(\bar\Phi_T^{n+1}\right).\\
\end{aligned}
\right.
\end{equation}
{\begin{remark}\label{remarkfast}
  Efficient computation of \eqref{fastETDRK} poses a great challenge, due to the need for fast evaluation of the products between $\{\Upsilon_i(L_{\kappa_{T},h}\tau), \Upsilon_i(L_{\kappa_{\sigma},h}\tau), \Upsilon_i(L_{\kappa_{M},h}\tau)\}_{i=0,1,2}$ and vectors. For large-scale problems, directly computing these products can be computationally expensive. In regular domains, the specific structure matrix allows the use of FFT-based fast algorithms, significantly improving computational efficiency.
  \end{remark} }

  {We represent \( \Phi_T^n, \Phi_N^n, \Phi_V^n, \Phi_{\sigma}^n, \Phi_M^n, \Theta^n \in \mathbb{R}^{N \times N} \) in matrix form with entries \( {\phi_T^n}_{i,j}\), \({\phi_N^n}_{i,j}\) \( {\phi_V^n}_{i,j}\), \({\phi_{\sigma}^n}_{i,j}\), \({\phi_M^n}_{i,j}\), \(\theta^n_{i,j}\),  and define the operators
\begin{equation}\label{L3}
  \mathcal{L}_{\kappa_{T}, h} = -\kappa_{T_1} \Delta_h + \kappa_{T_2} \Delta_h^2, \quad \mathcal{L}_{\kappa_{\sigma}, h} = -\frac{M_{\sigma}}{\delta_\sigma} \Delta_h + \kappa_{\sigma} \mathcal{I}, \quad \mathcal{L}_{\kappa_{M}, h} = -M_M \Delta_h + \kappa_{M} \mathcal{I} ,
  \end{equation}
whose matrix forms are given by \( L_{\kappa_{T}, h} \), \( L_{\kappa_{\sigma}, h} \), and \( L_{\kappa_{M},h} \), respectively. $\Delta_h$ is the discrete Laplace operator with Neumann boundary conditions, using the central finite difference approximation as described in Sec.~\ref{space_disc}. }

Since \eqref{L3} and the operators $\mathcal{L}_{\kappa_T, h}$, $\mathcal{L}_{\kappa_{\sigma}, h}$, and $\mathcal{L}_{\kappa_M, h}$ arise from the discretization of the operators $\mathcal{L}_{\kappa_T}$, $\mathcal{L}_{\kappa_{\sigma}}$, and $\mathcal{L}_{\kappa_M}$ with Neumann boundary conditions, these discrete operators are diagonalizable and have the following eigenvalue decompositions, as shown in \cite{du2019maximum, van1992computational}:
\begin{equation*}
\mathcal{L}_{\kappa_T, h} := \mathcal{C}^{-1} \Lambda_T \mathcal{C}, \quad
\mathcal{L}_{\kappa_{\sigma}, h} := \mathcal{C}^{-1} \Lambda_{\sigma} \mathcal{C}, \quad
\mathcal{L}_{\kappa_M, h} := \mathcal{C}^{-1} \Lambda_M \mathcal{C},
\end{equation*}
where $\mathcal{C}$ represents the 2D Discrete Cosine Transform (DCT) operator, and $\mathcal{C}^{-1}$ is the inverse 2D Discrete Cosine Transform (iDCT). For any $V=\left(V_{k, l}\right) \in \mathbb{C}^{(N+1) \times (N+1)}$, the following holds for the eigenvalue operations:
\begin{equation*}
\left(\Lambda_T V\right)_{k,l} = \lambda_{T, k, l} V_{k,l}, \quad
\left(\Lambda_{\sigma} V\right)_{k,l} = \lambda_{\sigma, k, l} V_{k,l}, \quad
\left(\Lambda_M V\right)_{k,l} = \lambda_{M, k, l} V_{k,l}, \quad 1 \leq k,l \leq N+1,
\end{equation*}
where $\lambda_{T, k, l}$, $\lambda_{\sigma, k, l}$, and $\lambda_{M, k, l}$ denote the eigenvalues of the operators $\mathcal{L}_{\kappa_T, h}$, $\mathcal{L}_{\kappa_{\sigma}, h}$, and $\mathcal{L}_{\kappa_M, h}$, respectively. These eigenvalues are computed as follows:
\begin{equation*}
\begin{aligned}
\lambda_{T, k, l} &= -\kappa_{T_1} \left( d_k^x + d_l^y \right) + \kappa_{T_2} \left( d_k^x + d_l^y \right)^2, \quad k,l = 1, 2, \dots, N+1, \\
\lambda_{\sigma, k, l} &= \kappa_{\sigma} - \frac{M_{\sigma}}{\delta_{\sigma}} \left( d_k^x + d_l^y \right), \quad k,l = 1, 2, \dots, N+1, \\
\lambda_{M, k, l} &= \kappa_M - M_M \left( d_k^x + d_l^y \right), \quad k,l = 1, 2, \dots, N+1,
\end{aligned}
\end{equation*}
where
\begin{equation*}
d_k^x = -\frac{4}{h^2} \sin^2 \left( \frac{(k-1) \pi}{2N} \right), \quad
d_l^y = -\frac{4}{h^2} \sin^2 \left( \frac{(l-1) \pi}{2N} \right), \quad k,l = 1, 2, \dots, N+1.
\end{equation*}

According to Lemma 2.2 (4) in \cite{du2019maximum}, for $i=0,1,2$, the following holds:
\begin{equation*}
\begin{aligned}
\Upsilon_i \left( \mathcal{L}_{\kappa_T, h} \tau \right) &= \mathcal{C}^{-1} \Upsilon_i \left( \Lambda_T \tau \right) \mathcal{C}, \quad
\left( \Upsilon_i \left( \Lambda_T \tau \right) V \right)_{k,l} = \Upsilon_i \left( \lambda_{T,k,l} \tau \right) V_{k,l}, \quad k,l = 1, 2, \dots, N+1, \\
\Upsilon_i \left( \mathcal{L}_{\kappa_{\sigma}, h} \tau \right) &= \mathcal{C}^{-1} \Upsilon_i \left( \Lambda_{\sigma} \tau \right) \mathcal{C}, \quad
\left( \Upsilon_i \left( \Lambda_{\sigma} \tau \right) V \right)_{k,l} = \Upsilon_i \left( \lambda_{\sigma,k,l} \tau \right) V_{k,l}, \quad k,l = 1, 2, \dots, N+1, \\
\Upsilon_i \left( \mathcal{L}_{\kappa_M, h} \tau \right) &= \mathcal{C}^{-1} \Upsilon_i \left( \Lambda_M \tau \right) \mathcal{C}, \quad
\left( \Upsilon_i \left( \Lambda_M \tau \right) V \right)_{k,l} = \Upsilon_i \left( \lambda_{M,k,l} \tau \right) V_{k,l}, \quad k,l = 1, 2, \dots, N+1.
\end{aligned}
\end{equation*}

The DCT and its inverse can be implemented using the 2D FFT through the following steps. Here, we focus on the computation in the x-direction, with the y-direction computation being analogous. First, we start with a vector \(\phi = [\phi_1, \phi_2, \ldots, \phi_{N+1}]^T\) of length \(N+1\), and define its reflection vector \(\hat{\phi} = [\phi_{N}, \phi_{N-1}, \ldots, \phi_2]^T\) of length \(N-1\). Next, apply a \(2N\)-point FFT to the combined vector $v = \left[\phi, \hat{\phi}\right]^T$, and extract the first $N+ 1$ components, dividing them by 2. For more detailed steps, refer to Sec.~4.4-4.5 in \cite{van1992computational}.

{
Thus, the overall computational complexity of the proposed schemes is reduced from \( O(N^3) \) to \( O(N^2 \log_2 N) \) per time step through FFT-based fast calculations in two dimension.}

{
\section{Numerical Simulations of Tumor Growth Model with ECM Degradation}}
This section presents the simulation results obtained using the previously introduced tumor growth model and the proposed numerical algorithm. {Simulations are performed with different initial ECM conditions, and both the temporal and spatial convergences are verified to demonstrate the effectiveness of the numerical algorithm. In addition, we simulate the growth of both aggressive and baseline tumors, discuss the impact of nutrient supply, the effect of varying MDE expression levels on ECM degradation, and the influence of different haptotaxis parameters on tumor growth.} Both 2D and 3D tumor growth models are considered, with parameter values specified in Table \ref{param}, including the total tumor, necrotic core, and viable tumor. Furthermore, the maximum bound principle (MBP) and bound-preserving properties are tested.


{
\subsection{Two-dimensional tumor growth simulations with ECM initial values surrounding the tumor tissue}\label{sec:num_validation}
}
First we focus on a 2D square domain $\Omega=(-1, 1)^{2}$. The initial condition for the total tumor volume fraction $\phi_T$ is given by a Gaussian function centered at the origin:
\begin{equation*}
\phi_{T,0} (x,y) =\exp \Big(1-\frac{1}{1-16\left(x^2 + y^2\right)} \Big).
\end{equation*}

Given that the volume fraction of the necrotic core is 0 at the initial time, the total tumor volume fraction is equivalent to the volume fraction of the viable tumor at the initial time. At the beginning of the simulation, the nutrient concentration is uniformly set to 1 throughout the domain, and the MDE volume fraction is initialized at 0. {For the initial ECM value, we refer to Fig.~2 in reference \cite{winkler2020concepts}, which discovers that the extracellular matrix (ECM) surrounds the tumor tissue, providing cellular support and forming a physical barrier. Then we set a similar initial condition for the ECM surrounding the tumor tissue. The initial ECM condition is defined by the following function:}
\begin{equation}\label{initial_ECM}
\theta_0 = \frac{1}{2} + \frac{1}{2} \left( 1 - 2 \left|\phi_{T,0} - 0.5\right| \right),
\end{equation}
This function provides a smooth transition from 1/2 to 1 at the edge of the ECM ring.

{
    Notice that, for the nutrient source term, in reference \cite{fritz2019local} (eq.~(2.7)), the source term simplifies to \(-\lambda_T^{\mathrm{pro}} \phi_V\) after substituting the simulation parameters, which does not depend on the nutrient concentration $\phi_\sigma$. We have chosen a concentration-dependent source term of the form $-\lambda_\sigma \phi_V \phi_\sigma$ in \eqref{phisigma}. When the nutrient concentration $\phi_\sigma$ is high, the nutrient consumption by viable tumor cells increases; whereas, at lower concentrations, the consumption decreases. This concentration dependence allows the source term to reflect the changes in nutrient consumption by viable tumor cells at different nutrient levels, dynamically describing the process of nutrient consumption by viable tumor cells.
}

The temporal convergence test was conducted with a fixed spatial resolution of \( m_x = n_y = N = 256 \), \(\tau = 2 \times 10^{-3}\), and time step sizes \(\tau/2, \tau/4, \tau/8, \tau/16\) are considered. Fig.~\ref{conv_tests}(\subref{time_convergence_test}) displays the variation of \( \phi_T \) in terms of \( x \), at \( y = 0 \), \( t = 7 \). The convergence with different time step sizes is clearly observed. {The spatial convergence test is performed with a fixed time step \( \tau = 1 \times 10^{-4} \), and the spatial mesh is refined from \( N=64\) to \( N=256\). Fig.~\ref{conv_tests}(\subref{space_convergence_test}) shows the variation of \( \phi_T \) in terms of \( x \), at \( y = 0 \), \( t = 7 \), confirming the spatial convergence.
}

{
Based on the validated numerical schemes from the time and space convergence tests, we now discuss the results of tumor growth simulations with the boundary conditions given by \eqref{ICBC} and the ECM initial conditions surrounding the tumor tissue.
}
\begin{figure}[t!]
  \centering
  \begin{subfigure}[t]{0.48\textwidth}
    \centering
    \includegraphics[width=\textwidth]{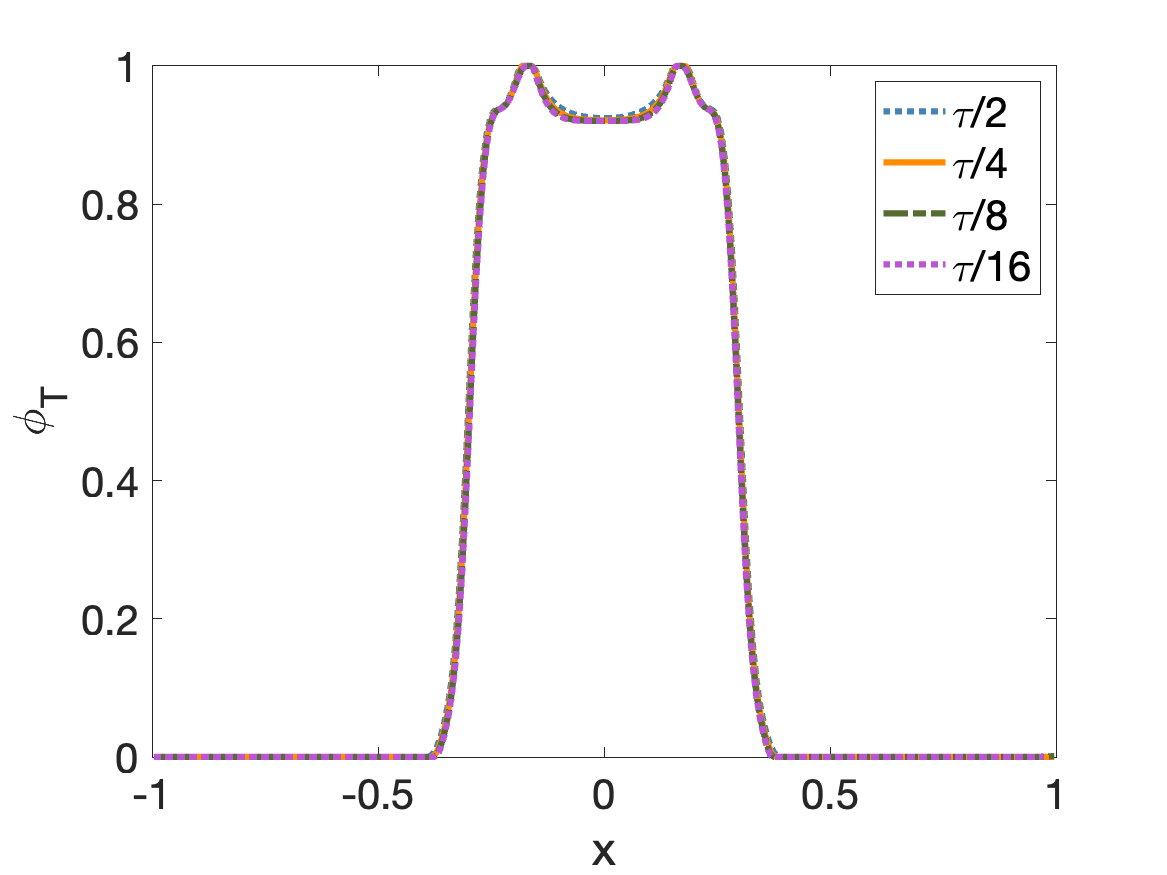}
    \caption{Time convergence test of the solution $\phi_T$: \( x \) at \( y = 0, t = 7 \) for different time step sizes.}
    \label{time_convergence_test}
  \end{subfigure}
  \hfill
  \begin{subfigure}[t]{0.48\textwidth}
    \centering
    \includegraphics[width=\textwidth]{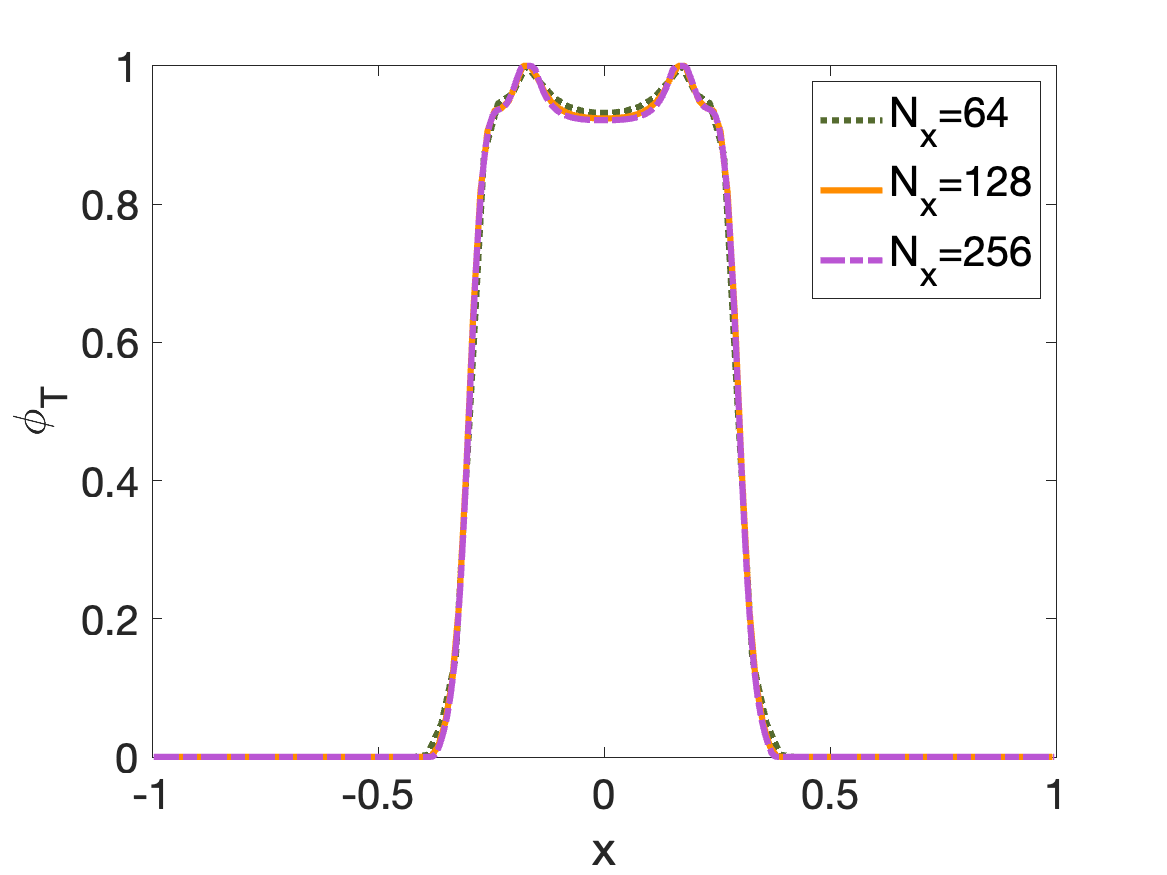}
    \caption{Space convergence test of the solution $\phi_T$:  \( x \) at \( y = 0, t = 7 \) for different spatial resolutions.}
    \label{space_convergence_test}
  \end{subfigure}
  \caption{Convergence test of the solution $\phi_T$}
  \label{conv_tests}
\end{figure}

{
  The simulation runs up to a final time of $T=10$. The time step size \(\tau\) is chosen based on the boundedness conditions outlined in Theorem~\ref{theoremBPN} and Theorem~\ref{theoremBPtheta}, which respectively require \(\tau \leq 2/\lambda_{VN}\) and \(\tau \leq 2/\lambda_{\theta}^{\text{deg}}\). These conditions are satisfied with the parameter values provided in Table \ref{param}. Time convergence tests (see Fig.~\ref{conv_tests}(\subref{time_convergence_test})) show that the numerical solution remains stable across different time step sizes. Therefore, we select \(\tau = 1 \times 10^{-3}\) as an appropriate time step size for the subsequent simulations. In terms of spatial discretization, a uniform grid is employed with \(m_x = n_y = N = 128\), ensuring adequate resolution to capture the dynamic evolution of the tumor interface. To guarantee the MBP for \(\psi_\sigma\) and \(\psi_M\), the conditions from Theorem~\ref{theoremMBPsigma} and Theorem~\ref{theoremMBPM} should be satisfied, which results in the parameter choices \(\kappa_{\sigma} \geq \lambda_{\sigma}\) and \(\kappa_M \geq \lambda_M^{\text{dec}} + \lambda_M^{\text{pro}}|\theta|_{\max} + \lambda_{\theta}^{\text{dec}}|\theta|_{\max}\).
}

\begin{figure}[t!]
    \centering
    \includegraphics[width=0.95\textwidth]{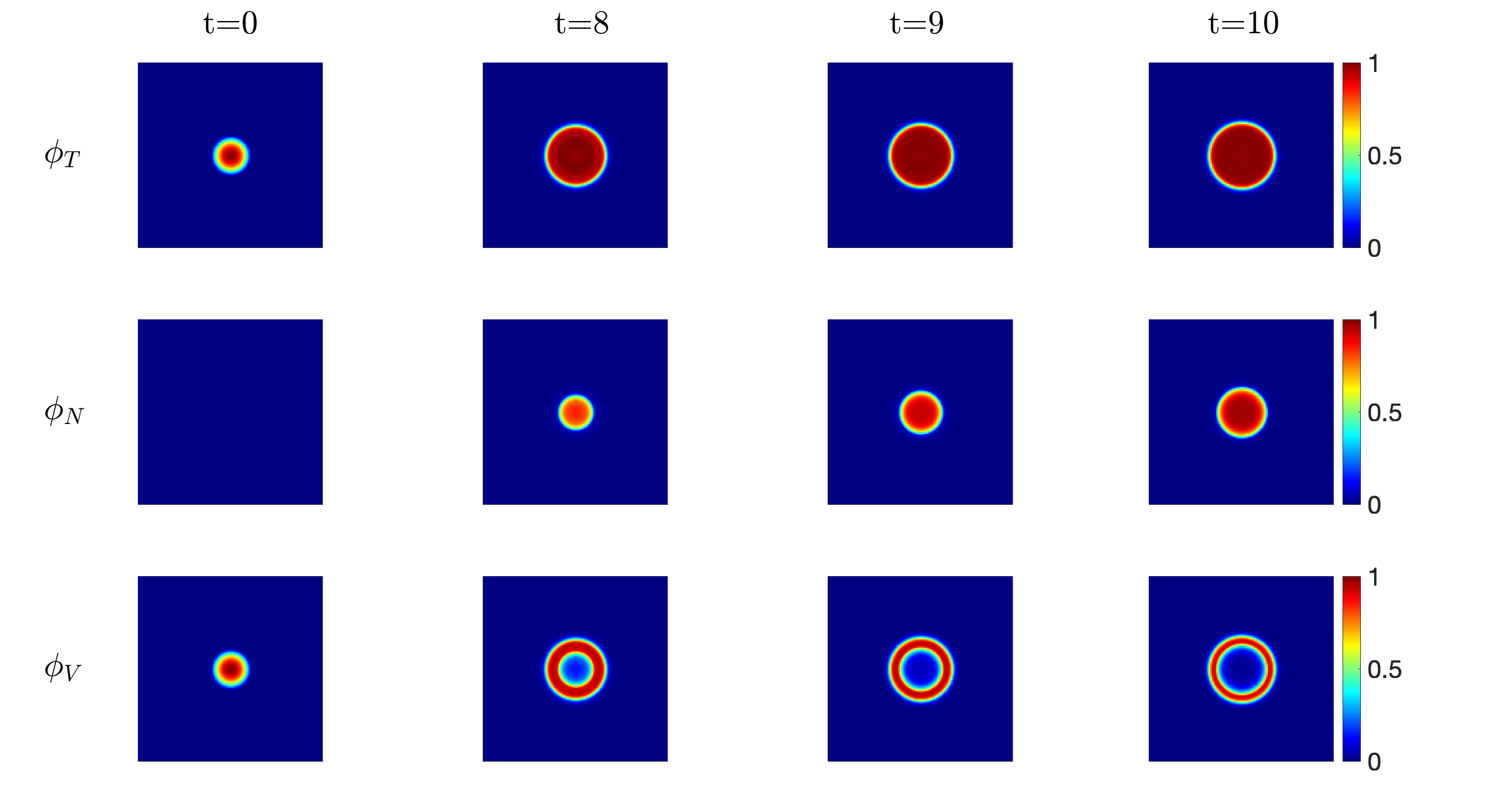}
    \caption{The volume fraction evolution of total tumor cells $\phi_T$, necrotic cells $\phi_N$, and viable cells $\phi_V$.}
    \label{tumor_2D}
\end{figure}
\begin{figure}[t!]
    \centering
    \includegraphics[width=0.85\textwidth]{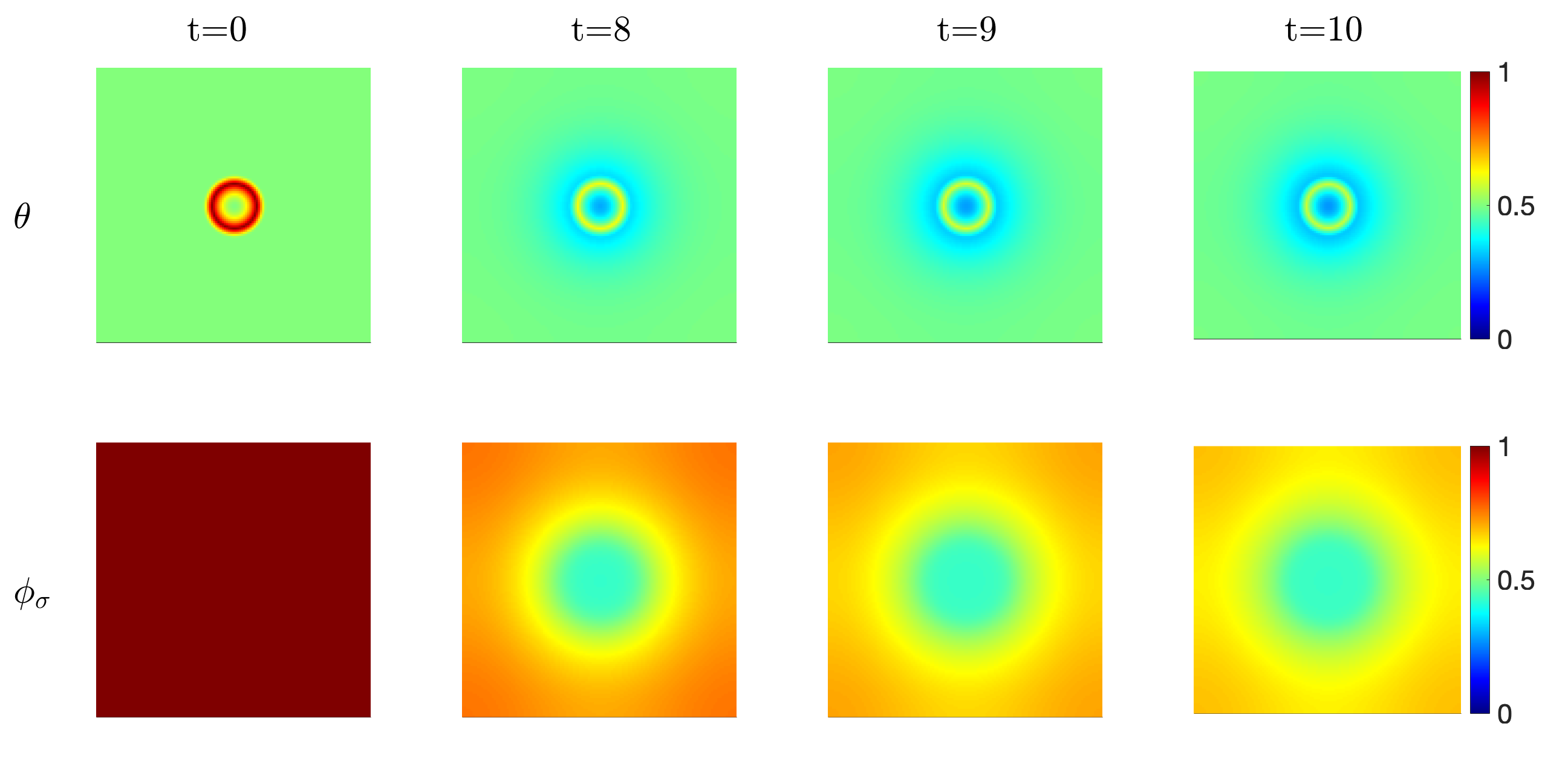}
    \caption{Evolution of ECM density and nutrient concentration in the tumor growth}
    \label{ECM_2D}
\end{figure}

Fig.~\ref{tumor_2D} illustrates the simulated volume fractions of total tumor cells $\phi_T$, necrotic cells $\phi_N$, and viable cells $\phi_V$ at time instants $t = 0, 8, 9, 10$. Initially, at $t=0$, there is no necrotic core. As time progresses, necrotic cells start to appear due to an insufficient supply of nutrients. The tumor forms a typical structure consisting of a ``necrotic core'' surrounded by an ``active shell''. Without the formation of new blood vessels (angiogenesis) to deliver extra oxygen and nutrients, the proportion of the necrotic core within the tumor will continue to increase over time.

In Fig.~\ref{ECM_2D}, the initial tumor is presented as a small circle, with the ECM distributed in a ring-like pattern around the tumor, forming a physical barrier.  As the tumor grows, the MDE released by viable tumor cells begin to degrade the ECM, which results in the degradation of the surrounding ECM of the tumor. While, the tumor continuously consumes nutrients, leading to a decline in the nutrient concentration. It is evident that the ECM and nutrient concentration are closely related to the tumor growth process, and their dynamic changes are crucial to the tumor microenvironment (TME).
\begin{figure}[t!]
  \centering

{\includegraphics[width=0.4\textwidth]{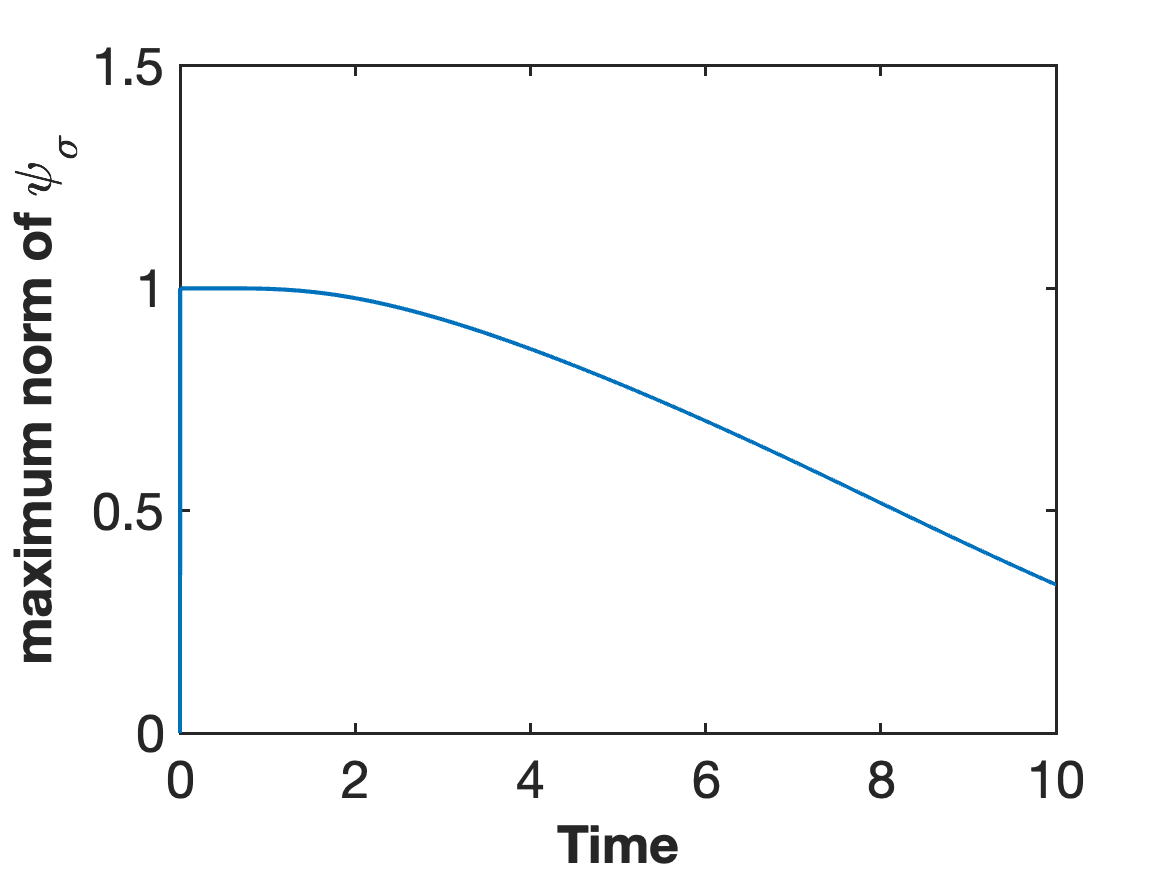}}
  {\includegraphics[width=0.4\textwidth]{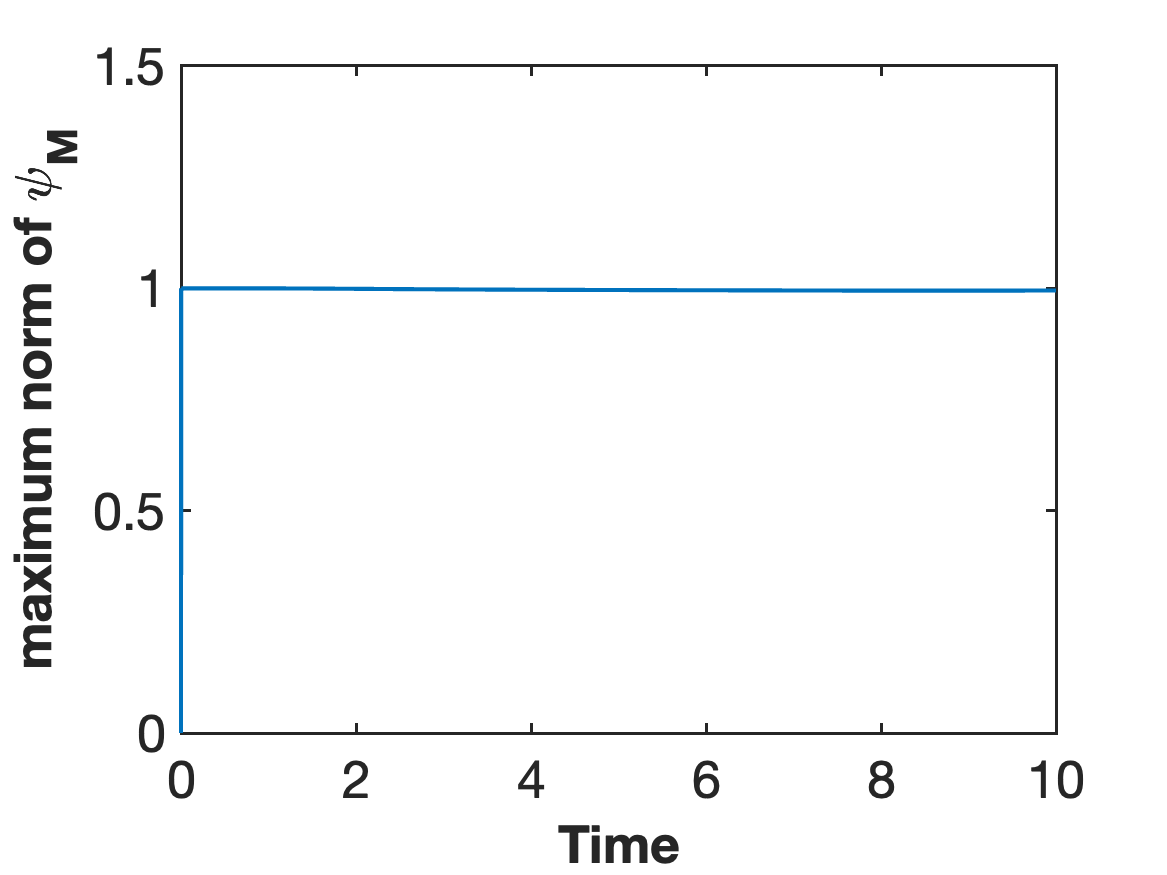}}
\caption{MBP test of $\psi_\sigma,$ and $ \psi_M$}
\label{MBPtest}
\end{figure}
\begin{figure}[t!]
  \centering
  {\includegraphics[width=0.32\textwidth ]{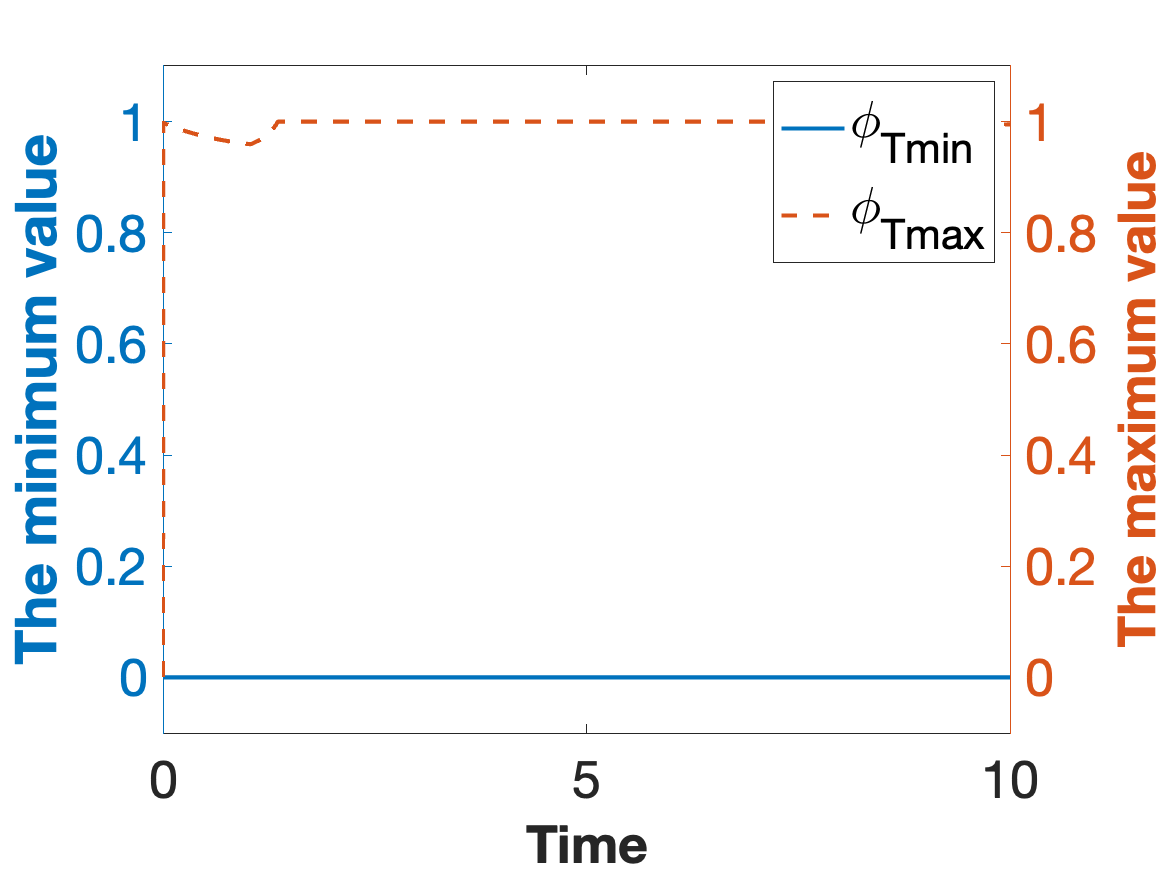}}
  {\includegraphics[width=0.32\textwidth ]{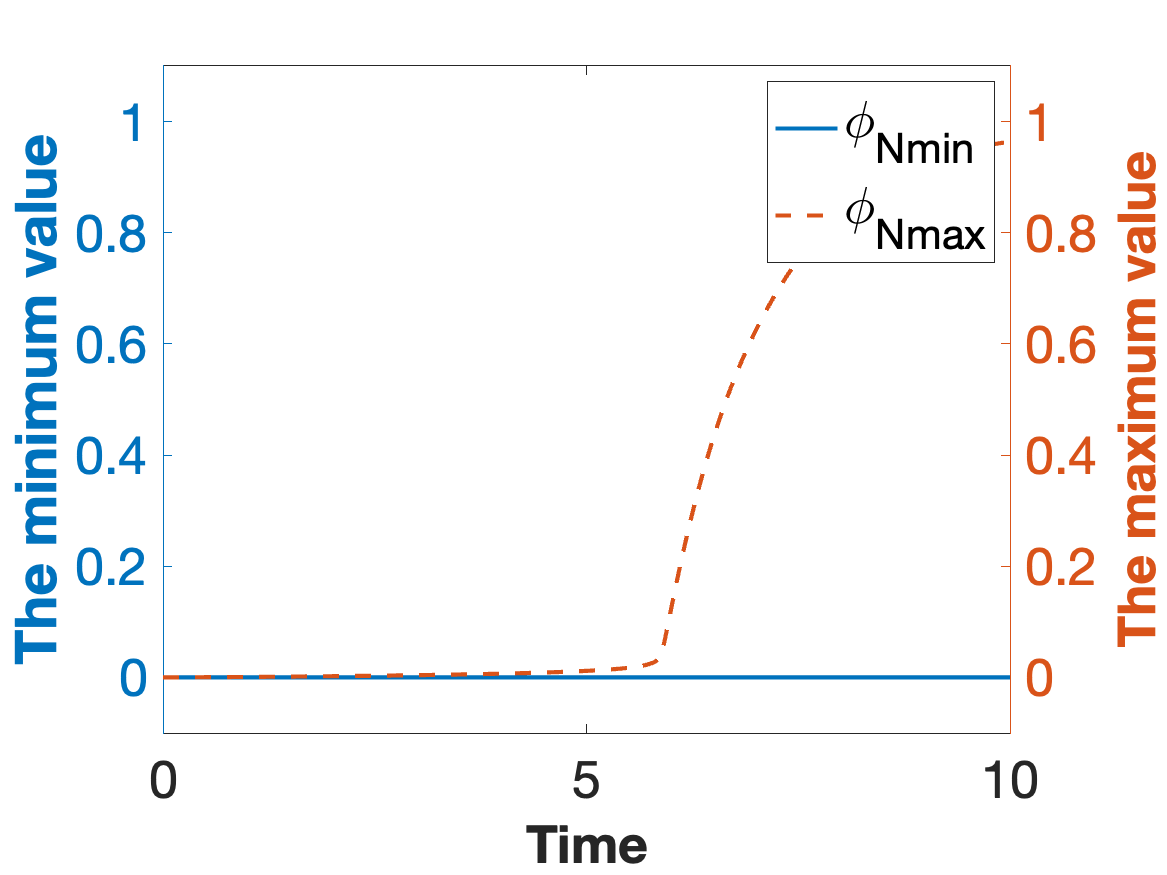}}
  {\includegraphics[width=0.32\textwidth ]{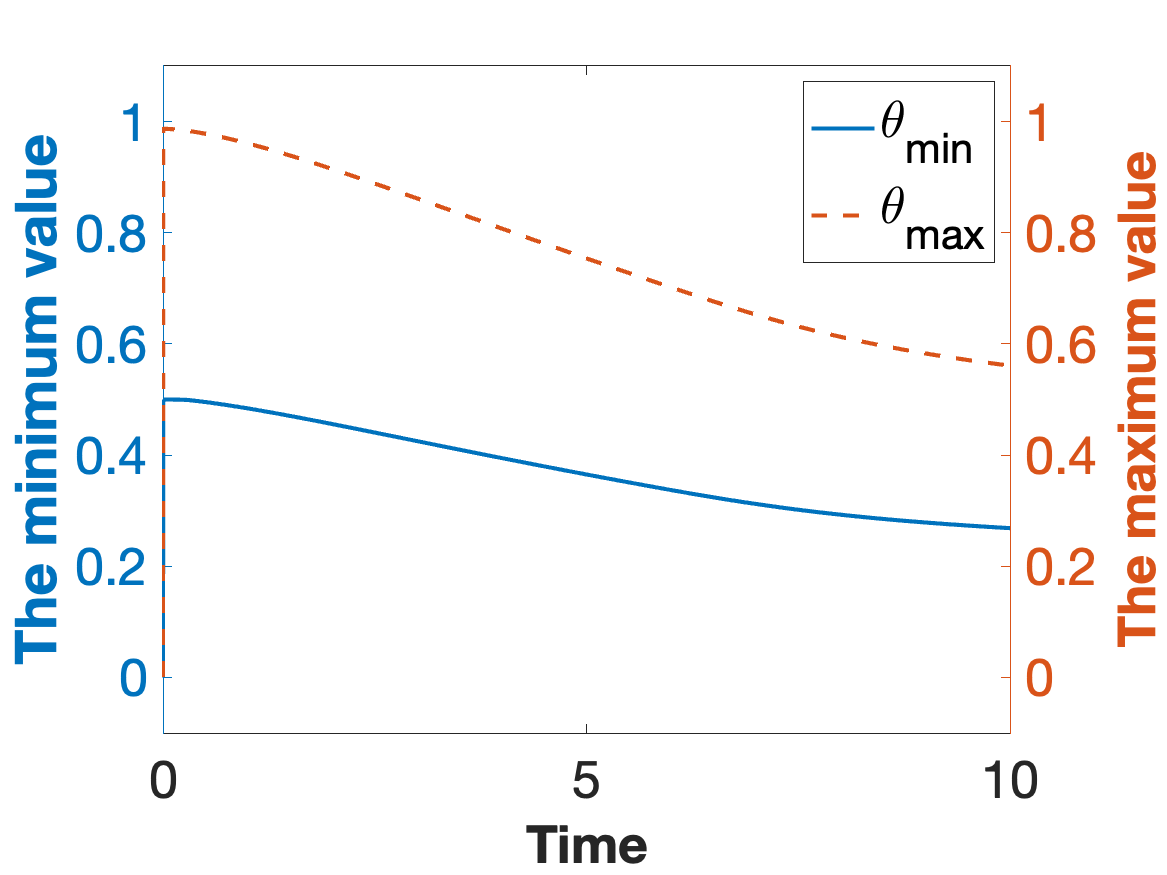}}
\caption{Bound preserving  test of $\phi_T, \phi_N,$ and $ \theta$}
\label{BPtest}
\end{figure}

We have also performed the MBP test for $\psi_\sigma$ and $\psi_M$, as well as the bound preserving test for $\phi_T$, $\phi_N$, and $\theta$, as illustrated in Figs.~\ref{MBPtest}-\ref{BPtest}. These tests have demonstrated  that the numerical algorithm can effectively maintain the inherent properties of the model, which are crucial for  accurately simulating tumor growth. This test further validates the effectiveness of our approach in handling complex biological processes.

{
  \subsection{Two-dimensional tumor growth simulations with quasi-homogeneous ECM initial values}
}

In the following simulations, we adopt the initial ECM setup from reference \cite{fritz2019local} to examine the interactions between the various biological factors under consideration.

The initial conditions for $\phi_T$, $\phi_N$, and $\phi_M$ are set as in the previous section. At the beginning of the simulation, the nutrient concentration is uniformly set to 1 throughout the domain. Similar to the original study \cite{fritz2019local}, a boundary condition of $\phi_\sigma=1$ was imposed on the right side of the domain to simulate a constant nutrient source. The MDE volume fraction is initialized at 0, and the ECM density is divided into two regions: one half of the domain is set to 1, and the other half to 1/2.

{
The tumor simulated with the parameters in Table \ref{param} is referred to as the baseline tumor. We first investigate the growth differences between the baseline tumor and aggressive tumors. Aggressive tumors may be represented by higher values of tumor proliferation and lower values of tumor apoptosis rate, for which the difference \( |\lambda_T^{\mathrm{pro}}| - |\lambda_T^{\mathrm{apo}}| \) becomes larger (we chose $\lambda_T^{\mathrm{pro}}=2.5$ and $\lambda_T^{\mathrm{apo}}=0.001$ in simulations).
}

\begin{figure}[t!]
  \centering
  \includegraphics[width=0.9\textwidth]{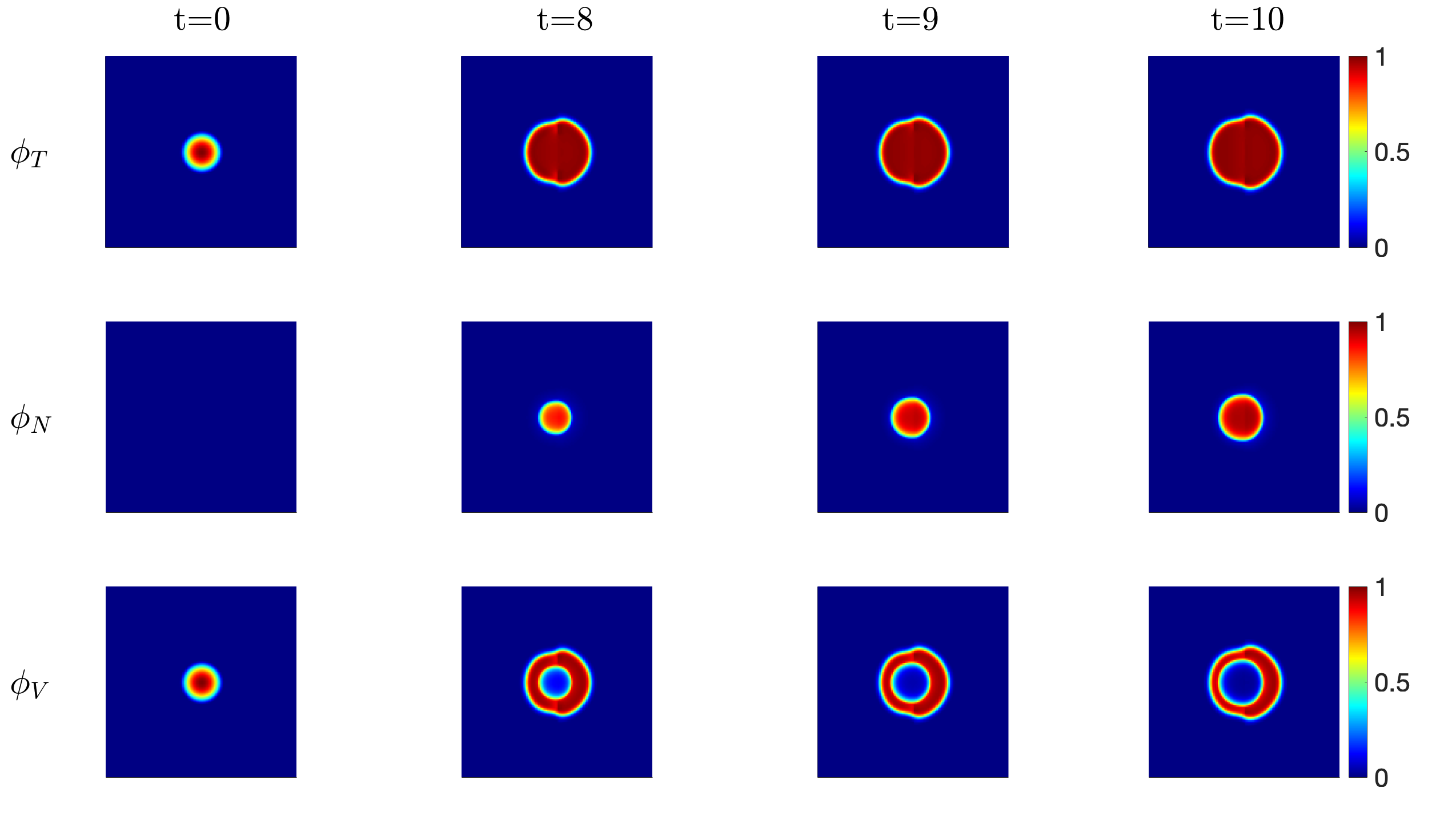}
  \caption{The evolution of the volume fractions of total tumor cells $\phi_T$ , necrotic cells $\phi_N$ , and viable cells $\phi_V$ for the baseline tumor.}
  \label{2D_tumor_base}
\end{figure}

\begin{figure}[t!]
  \centering
  \includegraphics[width=0.9\textwidth]{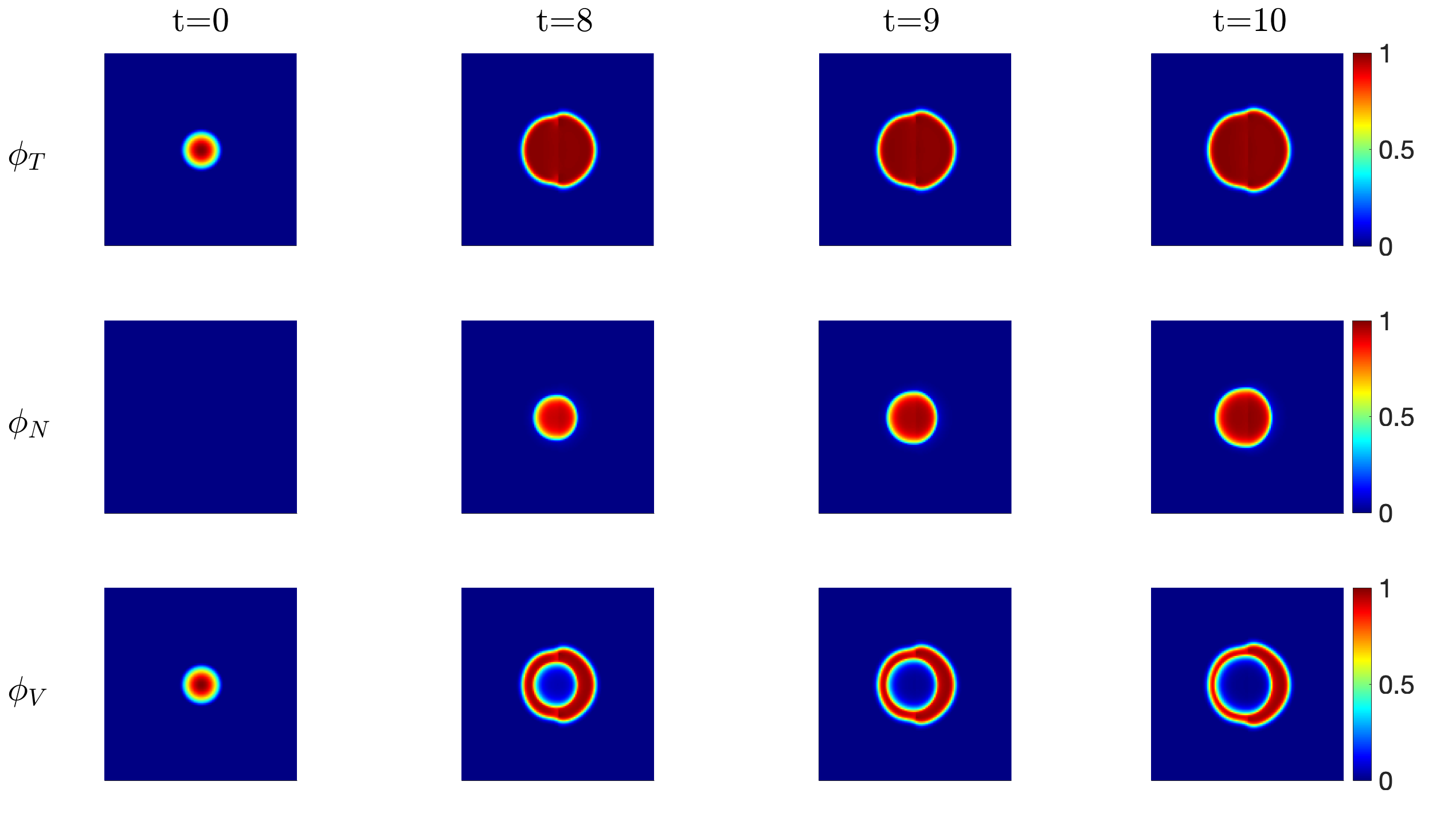}
  \caption{The evolution of the volume fractions of total tumor cells $\phi_T$ , necrotic cells $\phi_N$ , and viable cells $\phi_V$ for the aggressive tumor.}
  \label{2D_tumor_agg}
\end{figure}

\begin{figure}[t!]
  \centering
  \includegraphics[width=0.85\textwidth]{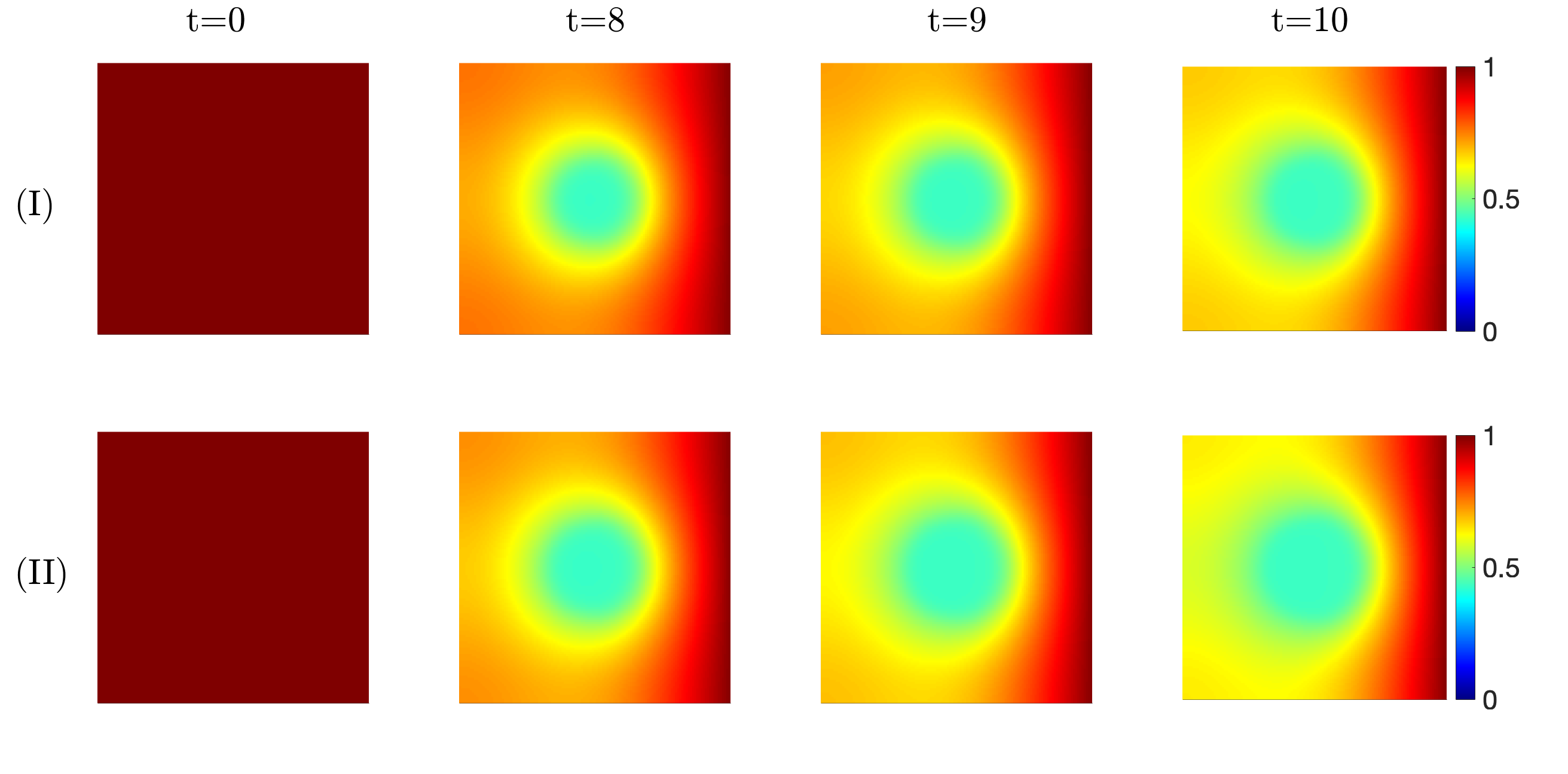}
  \caption{Nutrient concentration distribution at \( t = 0, 8, 9, 10 \) for (I) baseline and (II)aggressive tumors. }
  \label{2D_nutrient_base_agg}
\end{figure}

Fig.~\ref{2D_tumor_base} illustrates the  simulated volume fractions of total tumor cells $\phi_T$, necrotic cells $\phi_N$, and viable cells $\phi_V$ at time instants $t= 0,8,9,10$ for the baseline tumor. Initially, as shown in the first column of Fig.~\ref{2D_tumor_base}, {we begin with a small circular volume fraction of tumor cells} without any necrotic regions. Over time, as the tumor grows and nutrients struggle to reach the center, a necrotic core gradually forms. Viable tumor cells arrange themselves into a ring-shaped structure. As time progresses, the necrotic core expands, and the viable tumor cells continue to proliferate. By the time instant $t = 10$, the tumor exhibits a necrotic core surrounded by viable tumor cells. Furthermore, it can be observed that the tumor moves toward the right side, where constant nutrients are provided, simulating a high-nutrient concentration environment. Fig.~\ref{2D_tumor_agg} illustrates the growth of aggressive tumors, which exhibit similar structural features. In contrast to the baseline tumor, the aggressive tumor shows an expansion in volume, an increase in the necrotic core size, and a thinning of the ring structure of viable cells.

{In Fig.~\ref{2D_nutrient_base_agg}, the tumor continuously consumes the surrounding nutrients over time, reflecting the tumor's dependence on nutrient supply for growth. Meanwhile, compared to the baseline tumor, the aggressive tumor demonstrates faster nutrient depletion. Furthermore, Fig.~\ref{2D_tumor_base}-Fig.~\ref{2D_nutrient_base_agg} show that the tumor tends to expand into nutrient-rich areas to promote its growth and proliferation.}

{
  Reference \cite{fritz2019local} presents simulation results of a mixed structure where necrotic and viable cells occupy approximately half of the central volume each (\( \phi_N \approx \phi_V \approx 0.5 \)) over long-term estimation. In contrast, our simulation results indicate that the tumor undergoes stratification due to insufficient nutrient penetration as time progresses, ultimately forming a central necrotic core. As shown in Fig.~(3c) of Hirschhaeuser et al. \cite{hirschhaeuser2010multicellular}, the FaDu spheroid section exhibits \textquotedblleft the typical architecture of a multicellular tumor spheroid with an outer viable rim and a necrotic core\textquotedblright. Egloff-Juras et al. \cite{egloff2021validation} discovered that in FaDu spheroids, at day 3 the proliferative cells were present throughout the spheroid and by day 7, the necrotic core began to appear, \textquotedblleft while the proliferative cells were mainly located on the periphery of the spheroid, which can be considered a typical spheroid morphology \textquotedblright. Our simulation results have demonstrated behaviors similar to those described in these studies.
}

{
Next, we examine the effect of MDE expression on ECM degradation. In comparison with the baseline tumor, high MDE expression can be represented by higher MDE production rates and lower MDE decay rates, resulting in an increased difference \( |\lambda_M^{\mathrm{pro}}| - |\lambda_M^{\mathrm{dec}}| \). Specifically, the parameters for high MDE expression are given by \( \lambda_M^{\mathrm{pro}} = 1.5 \) and \( \lambda_M^{\mathrm{dec}} = 0.5 \). In contrast, low MDE expression is characterized by higher MDE decay rates and lower MDE production rates, leading to a smaller difference \( |\lambda_M^{\mathrm{pro}}| - |\lambda_M^{\mathrm{dec}}| \), with the parameters for low MDE expression being \( \lambda_M^{\mathrm{pro}} = 0.5 \) and \( \lambda_M^{\mathrm{dec}} = 1.5 \).
}
\begin{figure}[t!]
  \centering
  \includegraphics[width=0.9\textwidth]{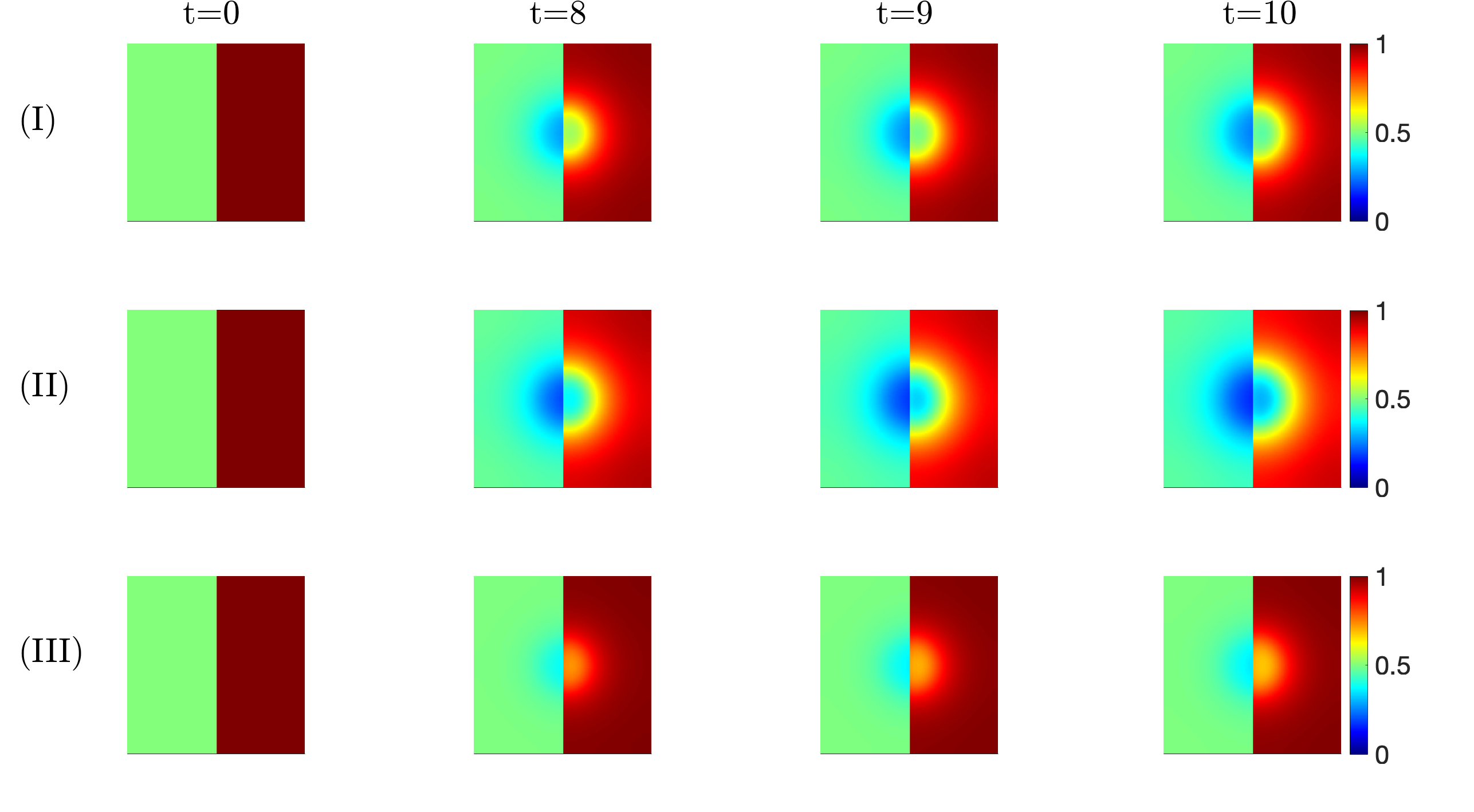}
  \caption{The evolution of ECM density during tumor growth under different levels of MDE expression: (I) baseline tumor, (II) high MDE expression, (III) low MDE expression}
  \label{2D_MDE_base_high_low}
\end{figure}
{
In Fig.~\ref{2D_MDE_base_high_low}, the evolution of ECM density during tumor growth under different levels of MDE expression is illustrated. The second and third rows imply that high MDE expression leads to more significant ECM degradation, while low MDE expression results in weaker ECM degradation. This dynamic interaction influences the ECM degradation process and affects the tumor microenvironment.
}

{
We now investigate the effects of different haptotaxis parameters \(\chi_H\) on tumor volume fraction growth. Fig.~\ref{2D_chiH_base_high_low} displays the tumor evolution under different haptotaxis parameters: a baseline tumor with \(\chi_H = 0.001\), a low haptotaxis parameter with \(\chi_H = 0.0005\), and a high haptotaxis parameter with \(\chi_H = 0.002\). Compared to the baseline tumor, the tumor with a low haptotaxis parameter adopts a more circular shape and exhibits lower sensitivity to the ECM gradient. In contrast, with a high haptotaxis parameter, the tumor shape forms a bump along the vertical axis. As shown in Fig.~\ref{2D_chiH_base_high_low}, with the increase in the haptotaxis parameter, tumor cells may become more sensitive to the ECM density gradient and could tend to migrate along regions with greater variations in ECM density.
}
\begin{figure}[t!]
  \centering
  \includegraphics[width=0.9\textwidth]{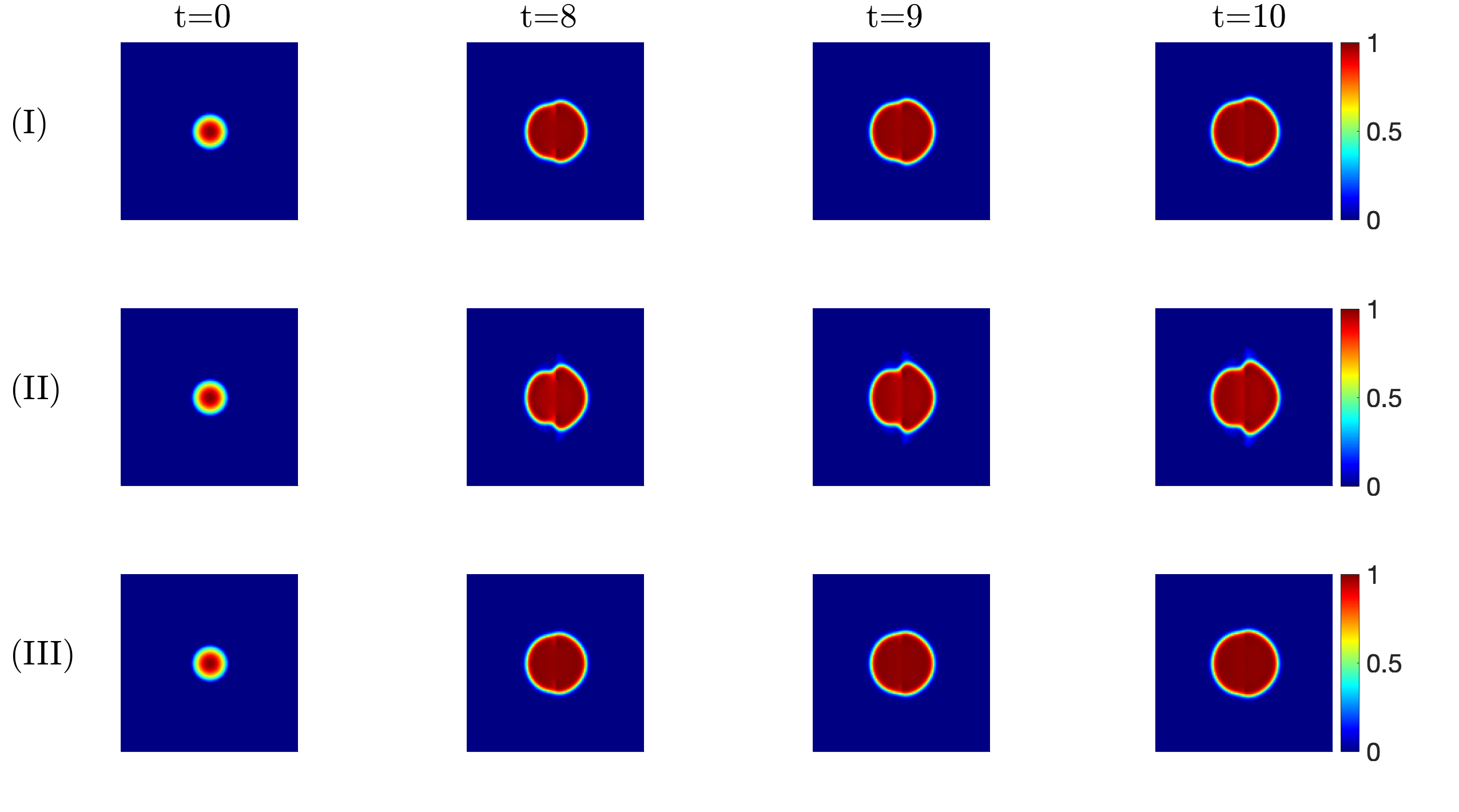}
  \caption{$\phi_T$ under different haptotaxis parameters: (I) baseline, (II) high Haptotaxis, (III) low haptotaxis parameters to ECM gradients}
  \label{2D_chiH_base_high_low}
\end{figure}

\subsection{Three-dimensional simulations}
In the three-dimensional computation, we consider two distinct, separated elliptical tumor volume fractions as initial conditions. In terms of the first tumor volume fraction, we define a spherical region centered at $(-0.15, -0.15, 0)$ and calculate the volume fraction using the following function:
\begin{equation*}
\phi_{T_{ball}} (x,y,z) = \exp \Big( 1-\frac{1}{1-16 ((x+0.15)^2 + (y+0.15)^2 + z^2 )} \Big).
\end{equation*}
Regarding the second tumor volume fraction, we rotate the coordinates, define an elliptical region centered at $(0.15, 0.15, 0)$, and calculate the volume fraction using the following function:
\begin{equation*}
\phi_{T_{elliptic}} (x,y,z) = \exp \Big(1-\frac{1}{1-16 ((x_{rot})^2 + (y_{rot}/1.35)^2 + z^2 )} \Big),
\end{equation*}
where $(x_{rot}, y_{rot})$ is obtained by rotating
$(x-0.15, y-0.15)$ using a rotation matrix defined as:
\begin{equation*}
\begin{bmatrix}
\cos(\gamma) & -\sin(\gamma) \\
\sin(\gamma) & \cos(\gamma)
\end{bmatrix}, \quad \mbox{with} \, \, \, \gamma = \frac{\pi}{4} .
\end{equation*}
Finally, we add these two tumor volume fractions to obtain the total tumor volume fraction
\begin{equation*}
\phi_T = \phi_{T_{ball}} + \phi_{T_{elliptic}}.
\end{equation*}

At the initial time, the nutrient concentration $\phi_\sigma$ throughout the domain is set to 1, the MDE volume fraction $\phi_M$ is initialized to 0, and the ECM density $\theta$ is also initialized using the ring structure, similar to the profile \eqref{initial_ECM} in the 2D simulation.

We take a numerical resolution $N_x=N_y=N_z=128$ and the time step size of $\tau=8 \times 10^{-3}$. Fig.~\ref{tumor_3D} illustrates the evolution of tumor cell volume fractions within the three-dimensional domain $\Omega = (-1, 1)^{3}$. These isosurfaces clearly depict the growth and expansion of the tumor over time. It is observed that the tumors begin to connect at $t = 2$; as time progresses, the conjoined tumor volume fraction becomes larger.
\begin{figure}[t!]
    \centering
\includegraphics[width=0.95\textwidth]{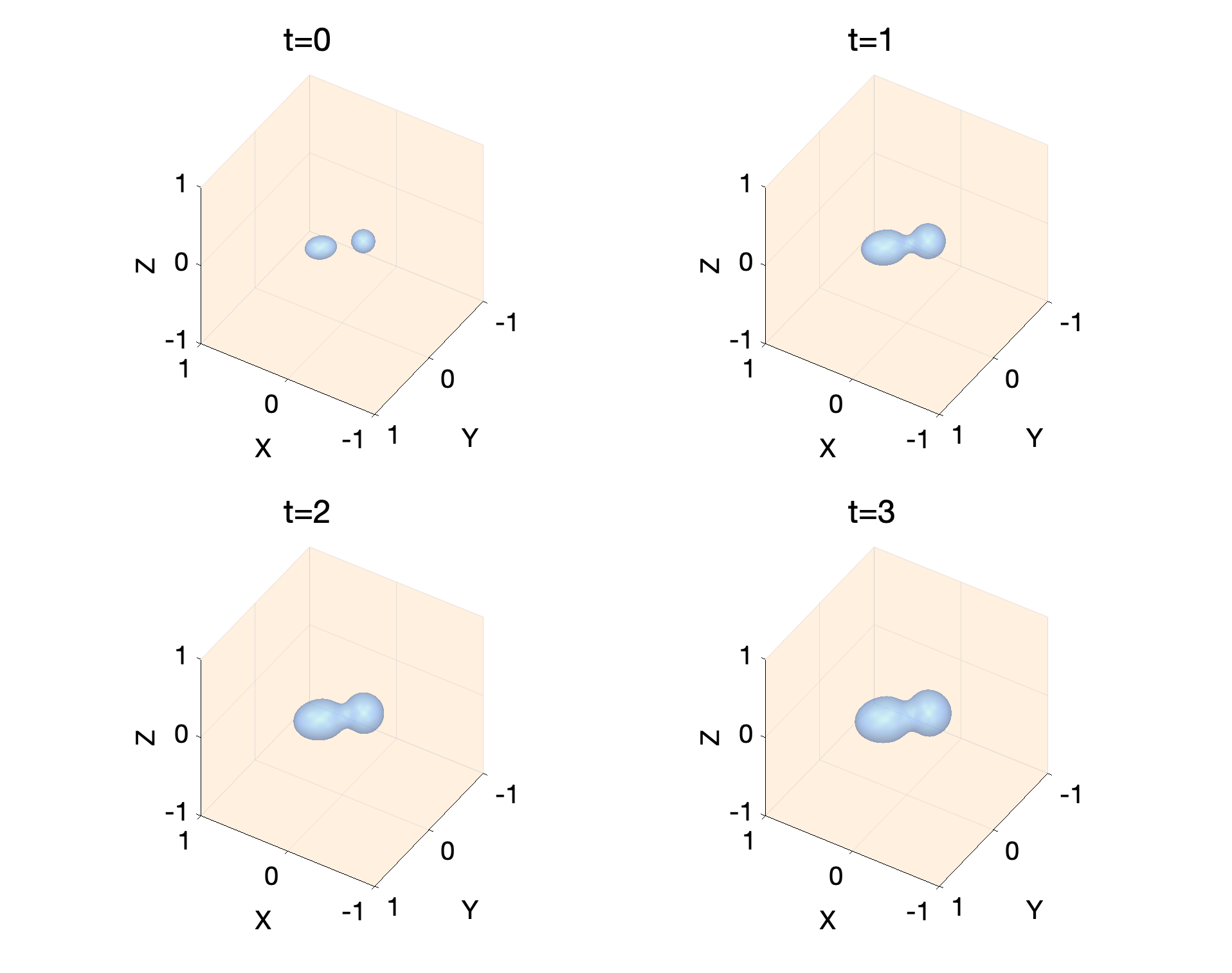}
    \caption{Volume fraction simulation of $\phi_T$ within a three-dimensional domain $\Omega=(-1, 1)^{3}$, with isosurfaces at a volume fraction of 0.8 shown at $t = \{0,1,2,3\}$.}
    \label{tumor_3D}
\end{figure}

{
  Our simulation results have provided visualizations of the tumor growth model. In addition, the presented numerical examples have also offered a clear understanding of how different biological factors influence tumor growth dynamics, and these results may assist in the prediction of tumor progression.
}

\section{Conclusion}
This research proposes a numerical approach to simulate a tumor growth model, emphasizing the degradation of the extracellular matrix (ECM). This model incorporates various biological elements, including tumor cells, viable cells, necrotic cells, and the evolution of MDE and ECM.

The main contributions of this work are fourfold. First, we have developed a decoupled, linear, explicit, and easily solvable algorithm. The algorithm exhibits excellent properties in terms of the MBP and bound preserving, with a theoretical proof. Second, the numerical algorithm employs a cutoff only for the variable $\phi_T$, while the remaining four variables automatically maintain the MBP and bound preserving properties. As a result, the computation adheres to the physical properties at the discrete level. In comparison, the method in an existing work~\cite{fritz2019local} applies a cutoff to all variables, which only guarantees that the variables lie within $[0,1]$ if they are inserted into the equation, but does not ensure the condition satisfied at every computed time steps. {Third, we simulated tumor growth under different initial ECM conditions and explored the effects of nutrient supply, variations in MDE expression levels on ECM degradation, as well as the influence of different chemotaxis parameters on tumor growth. Finally, for the ETDRK scheme, we accelerated the computation using FFT-based DCT.} In the presented 3D simulation ($N=128$, $\tau=8 \times 10^{-3}$, and $T=3$), we used a MacBook Pro 14 with an Apple M1 Pro chip and 16GB of RAM. This computation, implemented in MATLAB and completed in a total elapsed wall-clock time of 613.6921 seconds, has demonstrated high efficiency.

In future research, our objective is to explore the impact of angiogenesis on tumor growth and to consider specific cancers and drug therapies in the model. In conclusion, this research provides promising insights for the advancement of clinical research in cancer and enhances the understanding of tumor growth dynamics.

\section*{Acknowledgments}

Q. Huang's work is supported by the National Natural Science Foundation of China (No.12371385); Z. Qiao's work is partially supported by the CAS AMSS-PolyU Joint Laboratory of Applied Mathematics (No. JLFS/P-501/24) and the Hong Kong Research Grants Council RFS grant RFS2021-5S03, GRF grants 15302122 and 15305624; C. Wang's work is supported by the National Science Foundation (DMS 2012269, DMS 2309548).

\bibliographystyle{plain}
\bibliography{reference}

\begin{thebibliography}{10}

\bibitem{chaplain2011mathematical}
Mark A.~J. Chaplain, Mirosław Lachowicz, Zuzanna Szymanska, and Dariusz
  Wrzosek.
\newblock {Mathematical modelling of cancer invasion: the importance of
  cell--cell adhesion and cell--matrix adhesion}.
\newblock {\em Math. Models Methods Appl. Sci.}, 21(04):719--743, 2011.

\bibitem{chen2019positivity}
Wenbin Chen, Cheng Wang, Xiaoming Wang, and Steven~M Wise.
\newblock {Positivity-preserving, energy stable numerical schemes for the
  Cahn-Hilliard equation with logarithmic potential}.
\newblock {\em J. Comput. Phys. X}, 3:100031, 2019.

\bibitem{cox2021matrix}
Thomas~R Cox.
\newblock {The matrix in cancer}.
\newblock {\em Nat. Rev. Cancer.}, 21(4):217--238, 2021.

\bibitem{du2019maximum}
Qiang Du, Lili Ju, Xiao Li, and Zhonghua Qiao.
\newblock {Maximum principle preserving exponential time differencing schemes
  for the nonlocal {Allen--Cahn} equation}.
\newblock {\em SIAM J. Numer. Anal.}, 57(2):875--898, 2019.

\bibitem{du2021maximum}
Qiang Du, Lili Ju, Xiao Li, and Zhonghua Qiao.
\newblock {Maximum bound principles for a class of semilinear parabolic
  equations and exponential time-differencing schemes}.
\newblock {\em SIAM Rev.}, 63(2):317--359, 2021.

\bibitem{egloff2021validation}
Claire Egloff-Juras, Ilya Yakavets, Victoria Scherrer, Aur{\'e}lie Francois,
  Lina Bezdetnaya, Henri-Pierre Lassalle, and Gilles Dolivet.
\newblock Validation of a three-dimensional head and neck spheroid model to
  evaluate cameras for nir fluorescence-guided cancer surgery.
\newblock {\em Int. J. Mol. Sci.}, 22(4):1966, 2021.

\bibitem{feng2021maximum}
Jundong Feng, Yingcong Zhou, and Tianliang Hou.
\newblock {A maximum-principle preserving and unconditionally energy-stable
  linear second-order finite difference scheme for Allen--Cahn equations}.
\newblock {\em Appl. Math. Lett.}, 118:107179, 2021.

\bibitem{fritz2023tumor}
Marvin Fritz.
\newblock {Tumor evolution models of phase-field type with nonlocal effects and
  angiogenesis}.
\newblock {\em Bull. Math. Biol.}, 85(6):44, 2023.

\bibitem{fritz2021analysis}
Marvin Fritz, Prashant~K Jha, Tobias K{\"o}ppl, J~Tinsley Oden, and Barbara
  Wohlmuth.
\newblock {Analysis of a new multispecies tumor growth model coupling 3D
  phase-fields with a 1D vascular network}.
\newblock {\em Nonlinear Anal. Real World Appl.}, 61:103331, 2021.

\bibitem{fritz2019local}
Marvin Fritz, Ernesto~ABF Lima, Vanja Nikoli{\'c}, J~Tinsley Oden, and Barbara
  Wohlmuth.
\newblock {Local and nonlocal phase-field models of tumor growth and invasion
  due to ECM degradation}.
\newblock {\em Math. Models Methods Appl. Sci.}, 29(13):2433--2468, 2019.

\bibitem{fritz2019unsteady}
Marvin Fritz, Ernesto~ABF Lima, J~Tinsley~Oden, and Barbara Wohlmuth.
\newblock {On the unsteady Darcy--Forchheimer--Brinkman equation in local and
  nonlocal tumor growth models}.
\newblock {\em Math. Models Methods Appl. Sci.}, 29(09):1691--1731, 2019.

\bibitem{hartman2002ordinary}
Philip Hartman.
\newblock {\em {Ordinary Differential Equations}}.
\newblock SIAM, 2002.

\bibitem{hawkins2012numerical}
Andrea Hawkins-Daarud, Kristoffer~G van~der Zee, and J~Tinsley~Oden.
\newblock {Numerical simulation of a thermodynamically consistent four-species
  tumor growth model}.
\newblock {\em Int. J. Numer. Methods Biomed. Eng.}, 28(1):3--24, 2012.

\bibitem{henke2020extracellular}
Erik Henke, Rajender Nandigama, and S{\"u}leyman Erg{\"u}n.
\newblock {Extracellular matrix in the tumor microenvironment and its impact on
  cancer therapy}.
\newblock {\em Front. Mol. Biosci.}, 6:160, 2020.

\bibitem{hirschhaeuser2010multicellular}
Franziska Hirschhaeuser, Heike Menne, Claudia Dittfeld, Jonathan West, Wolfgang
  Mueller-Klieser, and Leoni~A Kunz-Schughart.
\newblock Multicellular tumor spheroids: an underestimated tool is catching up
  again.
\newblock {\em J. Biotechnol.}, 148(1):3--15, 2010.

\bibitem{huang2024maximum}
Qiumei Huang, Zhonghua Qiao, and Huiting Yang.
\newblock {Maximum bound principle and non-negativity preserving ETD schemes
  for a phase field model of prostate cancer growth with treatment}.
\newblock {\em Comput. Methods Appl. Mech. Engrg.}, 426:116981, 2024.

\bibitem{ju2021maximum}
Lili Ju, Xiao Li, Zhonghua Qiao, and Jiang Yang.
\newblock {Maximum bound principle preserving integrating factor Runge--Kutta
  methods for semilinear parabolic equations}.
\newblock {\em J. Comput. Phys.}, 439:110405, 2021.

\bibitem{LiD2017}
Dong Li, Zhonghua Qiao, and Tao Tang.
\newblock {Gradient bounds for a thin film epitaxy equation}.
\newblock {\em J. Differential Equations}, 262(3):1720--1746, 2017.

\bibitem{lima2014hybrid}
Ernesto~A.B.F. Lima, J.~Tinsley Oden, and Regina~C. Almeida.
\newblock {A hybrid ten-species phase-field model of tumor growth}.
\newblock {\em Math. Models Methods Appl. Sci.}, 24(13):2569--2599, 2014.

\bibitem{liu2022two}
Chaoyu Liu, Zhonghua Qiao, and Qian Zhang.
\newblock {Two-phase segmentation for intensity inhomogeneous images by the
  Allen-Cahn local binary fitting model}.
\newblock {\em SIAM J. Sci. Comput.}, 44(1):B177--B196, 2022.

\bibitem{lu2012extracellular}
Pengfei Lu, Valerie~M Weaver, and Zena Werb.
\newblock {The extracellular matrix: a dynamic niche in cancer progression}.
\newblock {\em J. Cell Biol.}, 196(4):395--406, 2012.

\bibitem{madsen2015source}
Daniel~H Madsen and Thomas~H Bugge.
\newblock {The source of matrix-degrading enzymes in human cancer: Problems of
  research reproducibility and possible solutions}.
\newblock {\em J. Cell Biol.}, 209(2):195--198, 2015.

\bibitem{mohammadi2019simulation}
Vahid Mohammadi and Mehdi Dehghan.
\newblock {Simulation of the phase field Cahn--Hilliard and tumor growth models
  via a numerical scheme: element-free Galerkin method}.
\newblock {\em Comput. Methods Appl. Mech. Engrg.}, 345:919--950, 2019.

\bibitem{nguyen2018mathematical}
Yen~T Nguyen~Edalgo and Ashlee~N Ford~Versypt.
\newblock {Mathematical modeling of metastatic cancer migration through a
  remodeling extracellular matrix}.
\newblock {\em Processes}, 6(5):58, 2018.

\bibitem{sfakianakis2020hybrid}
Nikolaos Sfakianakis, Anotida Madzvamuse, and Mark A.~J. Chaplain.
\newblock {A hybrid multiscale model for cancer invasion of the extracellular
  matrix}.
\newblock {\em Multiscale Model. Simul.}, 18(2):824--850, 2020.

\bibitem{sung2021global}
Hyuna Sung, Jacques Ferlay, Rebecca~L Siegel, Mathieu Laversanne, Isabelle
  Soerjomataram, Ahmedin Jemal, and Freddie Bray.
\newblock {Global cancer statistics 2020: GLOBOCAN estimates of incidence and
  mortality worldwide for 36 cancers in 185 countries}.
\newblock {\em CA Cancer J. Clin.}, 71(3):209--249, 2021.

\bibitem{van1992computational}
Charles Van~Loan.
\newblock {\em {Computational Frameworks for the Fast Fourier Transform}}.
\newblock SIAM, 1992.

\bibitem{winkler2020concepts}
Juliane Winkler, Abisola Abisoye-Ogunniyan, Kevin~J Metcalf, and Zena Werb.
\newblock Concepts of extracellular matrix remodelling in tumour progression
  and metastasis.
\newblock {\em Nat. Commun.}, 11(1):5120, 2020.

\bibitem{wise2010unconditionally}
Steven~M Wise.
\newblock {Unconditionally stable finite difference, nonlinear multigrid
  simulation of the Cahn-Hilliard-Hele-Shaw system of equations}.
\newblock {\em J. Sci. Comput.}, 44(1):38--68, 2010.

\bibitem{wu2014stabilized}
X~Wu, GJ~Van~Zwieten, and KG~Van~der Zee.
\newblock {Stabilized second-order convex splitting schemes for Cahn--Hilliard
  models with application to diffuse-interface tumor-growth models}.
\newblock {\em Int. J. Numer. Methods Biomed. Eng.}, 30(2):180--203, 2014.

\bibitem{xu2017full}
Jiangping Xu, Guillermo Vilanova, and Hector Gomez.
\newblock {Full-scale, three-dimensional simulation of early-stage tumor
  growth: the onset of malignancy}.
\newblock {\em Comput. Methods Appl. Mech. Engrg.}, 314:126--146, 2017.

\bibitem{zhang2023third}
Hong Zhang, Xu~Qian, and Songhe Song.
\newblock {Third-order accurate, large time-stepping and
  maximum-principle-preserving schemes for the Allen-Cahn equation}.
\newblock {\em Numer. Algorithms}, pages 1--38, 2023.

\bibitem{zhang2023temporal}
Hong Zhang, Jingye Yan, Xu~Qian, and Songhe Song.
\newblock {Temporal high-order, unconditionally maximum-principle-preserving
  integrating factor multi-step methods for Allen-Cahn-type parabolic
  equations}.
\newblock {\em Appl. Numer. Math.}, 186:19--40, 2023.

\bibitem{zheng2022tumor}
Xiaoming Zheng, Kun Zhao, Trachette Jackson, and John Lowengrub.
\newblock {Tumor growth towards lower extracellular matrix conductivity regions
  under Darcy’s Law and steady morphology}.
\newblock {\em J. Math. Biol.}, 85(1):5, 2022.

\end{thebibliography}
\end{document}